\documentclass[11pt]{amsart}
\usepackage[margin=2.8cm]{geometry}

\usepackage[british]{babel}
\usepackage{multicol}
\usepackage{array}
\usepackage{amsmath}
\usepackage{amssymb}
\usepackage{amsfonts}
\usepackage{amsthm}
\usepackage{multimedia}
\usepackage{float}
\usepackage{textcomp}
\usepackage{graphicx}
\usepackage{bbm}
\usepackage{mathtools}
\usepackage{enumerate}
\usepackage{framed} 
\usepackage{empheq}
\usepackage{fancyhdr}
\usepackage{xcolor}
\usepackage{appendix}

\newtheorem{theorem}{Theorem}

\newtheorem{proposition}[theorem]{Proposition}
\newtheorem{lemma}[theorem]{Lemma}

\newtheorem{remark}[theorem]{Remark}

\numberwithin{equation}{section}
\numberwithin{theorem}{section} 

\newcommand{\sign}{{\rm sign }}
\newcommand{\rd}{\mathrm{d}}

\newcommand{\RR}{\mathbb{R}}
\newcommand{\R}{\mathbb{R}}

\newcommand{\eps}{\varepsilon}

\newcommand{\dxt}{\dot{\bx}}

\newcommand{\ind}{\mathbf{1}}

\newcommand{\bx}{{\bf x}}

\let\oldhat\hat

\renewcommand{\hat}[1]{\oldhat{\mathbf{#1}}}

\numberwithin{equation}{section}

\begin{document}
\title{Nonlinear stability of chemotactic clustering with discontinuous advection}

\author{Vincent Calvez} \address{Institut Camille Jordan, CNRS UMR 5208, Universit\'e Claude Bernard Lyon 1, Universit\'e de Lyon, 69622 Villeurbanne, France}
\email{vincent.calvez@math.cnrs.fr}

\author{Franca Hoffmann} \address{Computing and Mathematical Sciences, California Institute of Technology, 1200 California Blvd, CA 91125 Pasadena, US.} \email{fkoh@caltech.edu}



\begin{abstract}
We perform the nonlinear stability analysis of a chemotaxis model of bacterial self-organization, assuming that bacteria respond sharply to chemical signals. The resulting discontinuous advection speed represents the key challenge for the stability analysis. We follow a perturbative approach, where the shape of the cellular profile is clearly separated from its global motion, allowing us to circumvent the discontinuity issue.
Further, the homogeneity of the problem leads to two conservation laws, which express themselves in differently weighted functional spaces. This discrepancy between the weights represents another key methodological challenge. We derive an improved Poincar\'e inequality that allows to transfer the information encoded in the conservation laws to the appropriately weighted spaces.
As a result, we obtain exponential relaxation to equilibrium with an explicit rate. A numerical investigation illustrates our results.  
\end{abstract}

\maketitle

\section{Introduction}

This work is devoted to the stability analysis of stationary clusters of bacterial cells under the effect of chemotaxis, as reported in the biophysical literature, see {\em e.g.} \cite{mittal_motility_2003}. 
Certain type of bacteria such as \emph{Escherichia coli} are sensible to chemical gradients and can communicate with each other by means of chemoattractants. In the absence of nutrients, the long-time asymptotics of the cell population at the macroscopic level depend on an intricate interplay between diffusive forces and chemotactic aggregation.
A classical model for bacterial motion is the Patlak--Keller--Segel model and variants where the chemotactic fluxes are derived analytically from a mesoscopic description of the dynamics at the individual level and possibly internal pathways \cite{HiOth00, MR2120548, MR2250124}. 
Here, we focus on a minimal model with discontinuous advection speed and, for the sake of simplicity, in one dimension of space (see \cite{saragosti_mathematical_2010} and discussion below). The cell density is denoted by $\rho(t,x)$, and the  chemoattractant concentration by $S(t,x)$. Cluster formation of cells in quasi-equilibrium can be modeled by a version of the classical Patlak-Keller-Segel model with linear diffusion and a discontinuous drift term,
\begin{subequations}
\label{quasimodel}
\begin{align}
& \partial_t \rho(t,x) =  \partial^2_{x} \rho(t,x) -  \chi \, \partial_x \left(\rho(t,x)  \sign(\partial_x S(t,x)) \right)\label{quasimodel1} \,,\\
& - \partial^2_{x} S(t,x) + \alpha S(t,x)  = \rho(t,x)  .\label{quasimodel2}
\end{align}
\end{subequations}
where $\chi>0$ is the chemoattractant sensitivity, and $\alpha\ge 0$ is the natural decay of the chemical.  The system \eqref{quasimodel} is equipped with an initial condition $\rho(0,x) = \rho_0(x)$, whose regularity and decay at infinity will be discussed below.

This model differs from the standard Patlak-Keller-Segel model for which the advection speed is essentially linear with respect to the chemical gradient, $\chi \partial_x S$ , see \cite{blanchet_parabolic-elliptic_2011,bellomo_toward_2015,painter_mathematical_2019} for recent reviews. The discontinuous nonlinearity $\chi \sign( \partial_x S)$ is the signature of  strong amplification of signal variations at the individual level. This is the extreme point of a family of chemotaxis models with non-linear dependency upon the chemical gradient that can be applied to shallow and steep gradients, see \cite{rivero_transport_1989} for the original derivation and \cite{ford_analysis_nodate} for a biological validation. We also refer to  \cite{hillen_users_2008} for a review on mathematical modeling of chemotaxis (see, in particular Model (M7)), and \cite{tindall_overview_2008} for a review dedicated to bacterial collective motion. The sign function was also used to reproduce traveling bands of bacteria with good agreement in  
\cite{saragosti_mathematical_2010}, which is the early motivation for the present work.

This strong discrepancy -- $\chi \sign( \partial_x S)$ versus $\chi \partial_x S$ -- requires a specific approach to handle the stability analysis of stationary states of \eqref{quasimodel}. There exists a family of stationary states to model~\eqref{quasimodel} given by
\begin{equation}\label{eq:statstate1}
\rho_\infty(x)
= \frac{M\chi}{2}e^{-\chi|x-a|}\, , \quad (M,a)\in \R_+^*\times \R\,. 
\end{equation}
The first paramater $M$ is linked to the conservation property of \eqref{quasimodel1}, and is fully determined by the mass of the initial condition, $M = \int \rho_0(x)\, dx$ since the equation for the cell density is conservative. The second parameter, $a$ is linked to the invariance by translation, and is also determined by the initial condition, though in a non-trivial way. It is interesting to notice that the natural length scale $L$ for the typical size of a cell cluster associated to the state \eqref{eq:statstate1} is $L = \chi^{-1}$, independent of the total mass of cells $M$. This is in agreement with the observations reported in \cite{mittal_motility_2003}, where the typical size of the cell cluster varies little with the number of bacteria. This is in opposition with the standard Patlak-Keller-Segel model, for which the length scale is $L = (\chi M)^{-1}$ due to the homogeneity of the problem (twice the number of cells results in twice the quantity of chemoattractant, and this in turn increases twofold the advection speed). 

In the present work, we assume $M = 2/\chi$ without loss of generality (in order to cancel the prefactor  in \eqref{eq:statstate1}). Further, we assume $\alpha>0$ as the case $\alpha=0$ can be treated in a more straighforward way with different methods, see Remark (5) following Theorem~\ref{thm:main}, or Appendix~\ref{sec:alpha0} for more details.

We stress that we are not concerned here with existence, uniqueness and regularity issues for solutions to \eqref{quasimodel}. We assume that solutions are sufficiently regular for our calculations to hold, that is, the cell density is continuous at all times $t\geq 0$, $\rho(t,\cdot) \in \mathcal C^0(\RR)$, while its derivative $\partial_x \rho$ develops jump discontinuities at points where $\partial_x S$ changes sign (see below for more details).

\bigskip

Our goal is to prove local stability for the nonlinear problem \eqref{quasimodel} around the family of stationary states \eqref{eq:statstate1}.
The difficulty here arises from the discontinuity of the advection speed $\chi \sign(\partial_x S)$.  
On the one hand, the choice of the sign function rules out a direct linearization of the non-linear term. On the other hand, \eqref{quasimodel1} is a piecewise linear advection-diffusion equation, up to the knowledge of those points where  $\partial_x S$ changes sign. We base our strategy on the latter observation. More precisely, we will crucially use the following preliminary result: 
\begin{proposition}\label{prop:H1}
If the initial density $\rho_0(x)$ is such that $S_0(x)$ has a unique critical point (which is a global maximum), then so is $S(t,x)$ at all times $t\ge 0$.
\end{proposition}

We prove this key result in Section \ref{sec:H1}. In this paper, we shall assume throughout that the condition of Proposition \ref{prop:H1} is fulfilled.  

\begin{remark}
One sufficient condition to ensure that  $S_0(x)$ has a unique critical point is that $\rho_0(x)$ itself has a unique critical point. However, this is not necessary.
\end{remark}

The key to our approach is to separate the question of the shape of the cell density profile from the movement of its center.  
Proposition \ref{prop:H1} enables to define uniquely the point $\bx(t)$ where $\partial_x S(t,x)$ changes sign: 
\begin{equation}\label{defX}
 \begin{cases}
 \sign(\partial_x S(t,x)) = 1\,, & \quad x<\bx(t)\,,\\
 \sign(\partial_x S(t,x)) = -1\,, & \quad x>\bx(t)\,.\\
 \end{cases}
 \end{equation} 
The dynamics of $\bx(t)$ inherit from the condition $\partial_x S(t,\bx(t)) = 0$, that is
\begin{equation*}
\frac{d}{dt}\left[ \partial_x S(t, \bx(t))\right] = 
\partial_{tx} S(t,\bx(t)) + \partial^2_{x} S(t,\bx(t)) \dxt(t) = 0\, .
\end{equation*}
The formula above can be combined with the following representation of $S = K_\alpha*\rho$ for $\alpha>0$, in order to derive a suitable expression for $\dxt$. Here, $K_\alpha$ is the fundamental solution of \eqref{quasimodel2}: 
\begin{equation}\label{eq:S conv}
S(t,x)
= \frac{1}{2\sqrt{\alpha}}
\int_{-\infty}^\infty  e^{- \sqrt{\alpha} |x-z|}  \rho(t,z) \, \rd z\, .
\end{equation}
This formulation enables to derive the following dynamics of $\dxt$ in the moving frame, as a function of the half derivatives of the cell density at the interface:
\begin{equation}\label{eq:dxt}
\dxt(t) = \frac{\partial_x  \rho(t,\bx(t)^+) + \partial_x \rho(t,\bx(t)^-)}{2 \partial^2_{x}  S(t,\bx(t))}\,.
\end{equation}
(see Lemma \ref{lem:xdot} for details).

It is then natural to reformulate \eqref{quasimodel} in the moving frame $y = x - \bx(t)$, writing for  $\tilde\rho(t,y) = \rho(t,x)$ and $\tilde S(t,y) = S(t,x)$:
\begin{subequations}
\label{quasimodelbis}
\begin{align}
& \partial_t \tilde\rho(t,y) =  \partial_{y}\left(\partial_y \tilde\rho(t,y) +(\chi \sign(y) + \dxt(t))\tilde \rho(t,y)\right)\, ,\label{quasimodel21}\quad \dxt(t) = \frac{\partial_y \tilde\rho(t,0^+) + \partial_y \tilde\rho(t,0^-)}{2 \partial^2_{y}  S(t,0)}\,,
 \\
& - \partial^2_{y} \tilde S(t,y) + \alpha \tilde S(t,y)  = \tilde\rho(t,y)  .\label{quasimodel22}
\end{align}
\end{subequations}
In doing so, we focus on the stability analysis of the shape of cell density $\tilde\rho$, and separate this question from the dynamics of the maximum point $\bx$ of $S$.

The stationary state for \eqref{quasimodelbis} is then simply given by
\begin{equation}\label{eq:stat state}
\tilde \rho_\infty(y)
=  e^{-\chi|y|}\, , \quad \dxt = 0\, . 
\end{equation}

We are now in a position to state our main result. Let $H^1_\chi$ be the weighted space of relative energy equipped with the following norm
\begin{equation}\label{eq:defH1}
    \|f\|_{H^1_\chi}
      =\left( \int_{-\infty}^\infty \left|\frac{f(y)}{\tilde\rho_\infty(y)}\right|^2 e^{-\chi|y|}\,\rd y+  \int_{-\infty}^\infty\left|\partial_y\left(\frac{f(y)}{\tilde\rho_\infty(y)}\right)\right|^2 e^{-\chi|y|}\,\rd y\right)^{1/2}
\,.  
\end{equation}

\begin{theorem}\label{thm:main}
Let $M = 2/\chi$ and $\alpha>0$. The family of stationary states    $\{ e^{-\chi|x-a|}\}_{a\in\R}$ to model \eqref{quasimodel} is locally nonlinearly stable in the following sense: There exists $\eps_0>0$ depending on $\chi$ and $\alpha$ such that, for all initial data satisfying 
$$\|\tilde \rho_0 - \tilde \rho_\infty\|_{H^1_\chi}\leq  \eps_0\,,$$
there exists a constant $C>0$, a rate $\gamma>0$ and a limit $\bx_\infty \in \R$ such that
\begin{equation}\label{eq:exp relax}
(\forall t\geq 0)\quad\|\tilde \rho(t,\cdot) - \tilde \rho_\infty\|_{H^1_\chi} \leq C e^{-\gamma t}\, , \quad \lim_{t\to \infty} \bx(t) = \bx_\infty\,,
\end{equation}
where $\bx(t)$ is as defined in \eqref{defX}, the constant $C$  depends on $\tilde \rho_0$,  $\eps_0$, $\chi$ and $\alpha$, and the rate $\gamma$ can be chosen arbitrarily below the following upper bound (at the expense of increasing the prefactor $C$ for larger choices of $\gamma$):
\begin{equation}\label{sigma}
\gamma < \gamma_0 =   \frac{\chi^2}{8} \left ( \frac{\chi + \sqrt{\alpha}}{\frac\chi2 + \sqrt{\alpha}} \right )\,.
\end{equation}
\end{theorem}

We make the following observations:
\begin{enumerate}
\item The limit $\bx_\infty$ exists but has no explicit value, up to our knowledge. Its  precise dependence on the initial configuration and parameters of the model is not known. Our analysis is not meant to derive a rate of convergence of $\bx(t) \to \bx_\infty$. 
\item The smallness condition $\|\tilde \rho_0 - \tilde \rho_\infty\|_{H^1_\chi}\leq  \eps_0$ does not control the initial value of $\dxt(0)$ which involves pointwise values $\partial_y\tilde\rho_0(0^\pm)$. Therefore, it is possible that $\bx(t)$ has large variations, meaning that $\bx(0)$ and $\bx_\infty$ are far apart. In fact, the convergence of $\bx(t)$ is obtained by means of the dissipative structure of the parabolic equation \eqref{quasimodel21}. 
\item  The upper bound of the convergence rate $\gamma_0$ in \eqref{sigma} is between $\frac{\chi^2}{8}$ ($\alpha\to\infty$) and $\frac{\chi^2}{4}$ ($\alpha\to 0$). However, our estimates on the convergence rate are not optimal, as resulting from successive inequalities. 
To obtain an optimal bound via functional inequalities, one could combine the improved Poincar\'e inequality \eqref{eq:poincare} with the interpolation inequality \eqref{lem:w0} into a single inequality and seek optimizers. We do not follow this approach here as the two separate inequalities contain meaningful structure.
\item The limit $\alpha\to \infty$ is quite singular, as it is formally equivalent to the following problem (after scaling of $S\mapsto S/\alpha$ and the identification in the vanishing viscosity limit $S = \rho$): 
\[\partial_t \rho(t,x) =  \partial^2_{x} \rho(t,x) -  \chi \, \partial_x \left(\rho(t,x)  \sign(\partial_x \rho(t,x)) \right)\, .\]
However, we have no insight about the above problem. Regarding \eqref{eq:exp relax}, it should be noted that $\gamma$ can be chosen independently of $\alpha$, but the prefactor $C$ becomes singular as $\alpha \to \infty$ in our methodology, due to the  control of non-linear contributions. 
\item The case $\alpha=0$ (excluded in the statement of the theorem) corresponds to settings where the degradation of the chemoattractant $S$ can be ignored. In this case and for more general and slightly smoother signal response functions, one can obtain global $L^1$-stability by reformulating model \eqref{quasimodel} as a scalar conservation law. This follows from an adaptation of the results in \cite[Chapter 7, Section 3]{Serre}, see Appendix~\ref{sec:alpha0}.
\end{enumerate} 

\subsection*{Method.} We emphasize one key methodological contribution. The problem \eqref{quasimodelbis} is equipped with two conservation laws, corresponding to the homogeneity of the problem \eqref{quasimodel1}, and its invariance by translation (see Section \ref{sec:CLs} for more details):
 \begin{subequations}
\label{eq:CS1}
\begin{align}
& \int_{-\infty}^\infty \tilde \rho(t,y)\, dy = \frac{\chi}{2}\, ,\\ 
& \int_{-\infty}^\infty  \partial_y \tilde \rho(t,y) e^{-\sqrt{\alpha}|y|} \, \rd y =0\,. \label{eq:CS1b}
\end{align}
\end{subequations}
Notably, there is a discrepancy between the weights,  here $1$ and $e^{-\sqrt{\alpha}|y|}$, turning into $e^{-\chi|y|}$ and $e^{-(\chi + \sqrt{\alpha})|y|}$ in relative energy, see \eqref{eq:defH1}. As a result, it is not obvious which is the correct choice of functional space to work in. The less restrictive option $e^{-(\chi + \sqrt{\alpha})|y|}$ might be considered. However, this results in a deterioration of the lower bound on the convergence rate in \eqref{sigma}. It is not even clear that all values of $\alpha$ can be encompassed. 
The alternative choice $e^{-\chi  |y|}$ turns out to be much more satisfactory. At the core of our method is an improved version of the standard Poincar\'e inequality with exponential weight, that allows to transfer the information given by the second conservation law to the appropriate weighted space, see Proposition \ref{lem:poincare}.  

\subsection*{Motivation and Perspectives.}
Our initial motivation comes from the mathematical modeling and analysis of concentration waves of chemotactic bacteria. We refer to \cite{tindall_overview_2008} for a comprehensive review of modeling of bacteria colonies. The present work is rooted in \cite{saragosti_mathematical_2010} where the following model was proposed for the propagation of density bands under the conjunct effect of two chemotactic signals (a communication signal $S$ secreted by the cells, and a nutrient signal $N$ depleted by the cells). 
\begin{align}
 &\partial_t \rho = D_\rho \partial^2_{x} \rho - \partial_x \left( \chi_S \rho \sign(\partial_x S) + \chi_N \rho \sign(\partial_x N) \right)\,,\notag \\
 &\partial_t S = D_S \partial^2_{x} S - \alpha S + \rho \,,\label{eq:model1}\\
 &\partial_t N = D_N \partial^2_{x} N - \gamma \rho N\,. \notag
\end{align}
The connection to \eqref{quasimodel} is the following. Firstly, in the absence of a food source, the nutrient signal $N$ is omitted, resulting in a zero speed traveling band, {\em i.e.} a stationary cluster, as reported in \cite{mittal_motility_2003}. Secondly, the communication signal $S$ is assumed to be in quasi-stationary equilibrium, that is \eqref{quasimodel2}. We also consider $D_\rho = D_S = 1$ without loss of generality (after appropriate non-dimensionalization), and we denote $\chi_S = \chi$ the chemosensitivity. 

The agreement with experimental data was shown to be very satisfactory \cite{saragosti_mathematical_2010}. Moreover, the model \eqref{eq:model1} was derived from a kinetic transport model suitable to describe the individual run-and-tumble motion of bacteria. The kinetic model agreed very well with another set of experimental data \cite{saragosti_directional_2011}. 

It is one of the key perspectives of this work to prove stability of the stationary cluster solution of the kinetic-transport equation (whose existence is a particular case of \cite{calvez_chemotactic_2019}). Note that exponential relaxation to equilibrium was established in \cite{calvez_confinement_2015} for the linear problem, where the communication signal $S$ is prescribed, using hypocoercive techniques. This was extended in \cite{mischler_linear_2016} to any dimension of space. 

A second key perspective consists in proving non-linear stability of the traveling concentration waves, solution of \eqref{eq:model1}, see \cite{saragosti_mathematical_2010} and \cite{calvez_chemotactic_2019} for existence of such waves, resp. for the parabolic problem and for the kinetic transport problem.

\section{Separating movement from shape}

From now on, we work in the frame centered at $\bx(t)$: $y = x - \bx(t)$. Also, we drop the symbol $\sim$  for notational convenience.

\subsection{Uniqueness of chemoattractant peak}\label{sec:H1}

In this section, we will show that the unique maximum property for the chemoattractant concentration $S$ is preserved by the flow, and provide a proof for the corresponding main result Proposition \ref{prop:H1}.

To show this, we "decouple" the dynamics to the right of the  maximum from the dynamics to the left. 
This can be done since in the moving frame, $\partial_y S(t,y)$ solves the following equation on the half-space:
\begin{equation}\label{S-halfspace}
    \begin{cases}
    -\partial^2_{y}\left (\partial_y S(t,y)\right )+\alpha \partial_yS(t,y)=\partial_y\rho(t,y)\,, &y>0\,,\\
    \partial_y S(t,0)=0\,,
    \end{cases}
\end{equation}
where the source term $\partial_y\rho(t,y)$ is given via the solution to the following PDE:
\begin{equation}\label{model-yframe-3}
    \partial_t\left ( \partial_y\rho(t,y)\right ) =\partial^2_y\left(\partial_y\rho(t,y)+c(t)\rho(t,y)\right)\,,\quad 
    c(t):=\chi+\dxt(t) \,,\quad 
    y>0 \,.
\end{equation}
We denote by $A$ the fundamental solution of \eqref{S-halfspace}, such that 
\begin{equation}\label{eq:dyS}
\partial_y S(t,y)=\int_0^\infty A(y,z)\partial_z \rho(t,z)\, \rd z\,.
\end{equation}
Namely, we have (although the exact expression can be omitted)
\begin{equation*}
    A(y,z):=
    \begin{cases}
    \frac{1}{\sqrt{\alpha}}\sinh(\sqrt{\alpha}y) e^{-\sqrt{\alpha}z}\,,&y<z\,,\\
    \frac{1}{\sqrt{\alpha}}\sinh(\sqrt{\alpha} z) e^{-\sqrt{\alpha}y}\,,&y>z\,.
    \end{cases}
\end{equation*}

The following reformulation is of key importance to control the sign of $\partial_y S$ on $\R_+$. 
\begin{lemma}\label{lem:S'}
Let $S$ be a solution to \eqref{S-halfspace}. Then $\partial_{y} S$ satisfies the following boundary value problem:
\begin{equation}\label{eqn:S'}
 \begin{cases}    
\displaystyle  \partial_t\left ( \partial_y S(t,y)\right ) 
=\partial^3_{y}S(t,y)-\chi\alpha e^{-\sqrt{\alpha}y}S(t,0)
+\frac{\partial^2_{y}S(t,y) }{\partial^2_{y}S(t,0) }\left( \chi\alpha S(t,0) -\partial^3_{y}S(t,0^+) \right)\,, & y>0\,,\\
\partial_y S(t,0) = 0\, . 
\end{cases}
\end{equation}
\end{lemma}

\begin{proof}
By differentiating \eqref{eq:dyS}, we obtain
\begin{equation*}
  \partial_t\left (\partial_y S(t,y)\right ) 
= \int_0^\infty A(y,z)\partial_t\left (\partial_z \rho(t,z)\right )\, \rd z
  = \int_0^\infty A(y,z)\left ( \partial^3_z\rho(t,z)+c(t)\partial^2_z \rho(t,z)\right) \, \rd z\,.
\end{equation*}
By definition of $A$, the first contribution $S_3 := \int_0^\infty A(y,z)  \partial^3_z\rho(t,z)\rd z$ is such that 
\begin{equation*}
    \begin{cases}
    -\partial^2_{y}S_3(t,y)+\alpha S_3(t,y)=\partial^3_y\rho(t,y)\,, &y>0\,,\\
     S_3(t,0)=0\,.
    \end{cases}
\end{equation*}
On the other hand, we have
\begin{equation*}
    -\partial^2_{y}\left ( \partial_y^3S(t,y)\right )+\alpha \partial_y^3S(t,y)=\partial^3_y\rho(t,y)\,, \quad y>0\,.
\end{equation*}
We deduce that $R_3 := S_3 - \partial_y^3S$ satisfies the following boundary value problem:
\begin{equation*}
    \begin{cases}
    -\partial^2_{y}R_3(t,y)+\alpha R_3(t,y)=0\,, &y>0\,,\\
     R_3(t,0)=- \partial_y^3S(t,0^+)\,.
    \end{cases}
\end{equation*}
Its solution is $R_3(t,y) = - \partial_y^3S(t,0^+)e^{-\sqrt{\alpha}y}$ (the exponentially growing mode can be easily ruled out in the energy space \eqref{eq:defH1} by \eqref{S-halfspace}), hence
$$S_3(t,y) = \partial_y^3S(t,y) - \partial_y^3S(t,0^+) e^{-\sqrt{\alpha}y}\,.$$ 
Similarly, we find that the second contribution $S_2 :=  \int_0^\infty A(y,z)  \partial^2_z\rho(t,z)\rd z$ is 
$$S_2(t,y) = \partial_y^2S(t,y) - \partial_y^2S(t,0) e^{-\sqrt{\alpha}y} \,.$$ Thus, 
\begin{align*}
\partial_t\partial_y S(t,y) 
&= S_3(t,y)+c(t) S_2(t,y)\\
&=  \partial_y^3S(t,y) - \partial_y^3S(t,0^+) e^{-\sqrt{\alpha}y}  + c(t) \left ( \partial_y^2S(t,y) - \partial_y^2S(t,0) e^{-\sqrt{\alpha}y} \right )\, .
\end{align*}
The value of $c(t)$ depends on the dynamics of $\dxt$, see \eqref{quasimodel21}.  We develop its expression explicitly, focusing on the right hand side by using the continuity of the flux at the interface: $\partial_y \rho(0^+) + \chi \rho(0) = \partial_y \rho(0^-) - \chi \rho(0)$:
\begin{align*}
    c(t)
    &= \chi+\frac{\partial_y \rho(0^+) + \partial_y \rho(0^-)}{2 \partial^2_{y}S(0)}
    = \chi+\frac{\partial_y\rho(0^+) + \chi \rho(0)}{ \partial^2_{y}  S(0)}\\
&= \frac{\chi \partial^2_{y}  S(0)
-\partial^3_{y}  S(0^+) +\alpha \partial_y S(0) -\chi\partial^2_{y}  S(0)+\chi\alpha S(0)}{ \partial^2_{y}  S(0)}\\
& = \frac{\chi\alpha S(0)-\partial^3_{y}  S(0^+)}{ \partial^2_{y}  S(0)}\,,
\end{align*}
where the second line follows from \eqref{quasimodel22} and \eqref{S-halfspace}.
Substituting this expression into the one above, we obtain the result.
\end{proof}

\begin{proof}[Proof of Proposition~\ref{prop:H1}]
For $\eps>0$, define
$$S^\eps(t,y):=S(t,y)-\eps y \,, \quad \text{ then }\, \partial_y S^\eps(t,y)=\partial_y S(t,y)-\eps\, . $$
In particular, we have $\partial_y S^\eps(t,0) = -\eps <0$ for all $t>0$, by definition, and $\partial_y S^\eps(0,y)\leq - \eps <0$ for all $y>0$, by assumption.
For any fixed $\eps>0$, we aim to prove that $\partial_y S^\eps(t,y)<0$ for all $t,y>0$. Let us suppose by contradiction that there  exists a smallest time $t_0>0$ and closest location $y_0>0$ such that $\partial_yS^\eps(t_0,y_0)=0$. It is necessarily an interior maximum value of $\partial_yS^\eps$ with respect to $y$, so that the following conditions are verified:
\begin{equation*}
\partial^2_{y}S^\eps(t_0,y_0)=0\,, \quad \partial^3_{y}S^\eps(t_0,y_0)\le0\, ,
\end{equation*}
and, in addition, $\partial_t \partial_y S^\eps(t_0,y_0)\ge0$. We deduce from the PDE satisfied by $\partial_y S$ \eqref{eqn:S'} that $S(t_0,0) \leq 0$, which is clearly a contradiction as $S=K_\alpha\ast \rho>0$.

By letting $\eps \to 0$, we conclude that $\partial_y S(t,y) \leq 0$ for all $t,y>0$. 
Then, we can use the PDE \eqref{eqn:S'} again to show that the latter inequality is strict: in fact, it is a drift-diffusion equation on $\partial_y S(t,y)$ with negative source term and non-positive initial data. We just proved that the solution remains non-positive. By application of the strong maximum principle, it cannot have an interior maximum point, so $\partial_yS(t,y)$ is negative for $t,y>0$.  

A symmetric argument to the one presented above shows that $\partial_y S(t,y) > 0$ for $t>0$ and $y<0$, which concludes the proof of Proposition~\ref{prop:H1}. 
\end{proof}

\subsection{Dynamics of the chemoattractant peak}

The dynamics of the point $\bx(t)$ are given by the following lemma. 

\begin{lemma}\label{lem:xdot} We have
\begin{align}\label{xdot-main}
 \dxt(t) = \frac{\partial_y \rho(t,0^+) + \partial_y \rho(t,0^-)}{2 \partial^2_{y}  S(t,0)} .
\end{align}
\end{lemma}

\begin{proof} The maximum point $\bx(t)$ is defined implicitly as
\begin{align}
0 &= \partial_y S(t,0)
= \int_{-\infty}^\infty  \partial_y K_\alpha (-y)  \rho(t,y) \, \rd y\notag\\
&=-\frac{1}{2} \int_\RR \sign(-y) e^{- \sqrt{\alpha} |y|} \rho(t,y) \, \rd z.\label{dxS}
\end{align}
This is actually equivalent to the second conservation law \eqref{eq:CS1b} after integration by parts. 
Differentiating with respect to time  and substituting \eqref{quasimodel21}, we obtain
\begin{align*}
0 &= \frac{1}{2} \int_{-\infty}^\infty   \sign\left( y\right)  e^{- \sqrt{\alpha} |y|} \partial_t \rho (t,y) \, \rd y\label{xdot1}\\
& =  \frac{1}{2} \int_{-\infty}^\infty   \sign\left( y\right)  e^{- \sqrt{\alpha} |y|}   \partial_{y}\left(\partial_y \rho(t,y) +(\chi \sign(y) + \dxt(t)) \rho(t,y)\right)\, \rd y \\
& =  j(t,0)  - \frac{\sqrt{\alpha}}{2} \int_{-\infty}^\infty     e^{- \sqrt{\alpha} |y|} j(t,y)\, \rd y\,,
\end{align*}
where the flux 
\begin{equation}\label{flux}
   j(t,y) := - \left(\partial_y \rho(t,y) +(\chi \sign(y) + \dxt(t)) \rho(t,y)\right) 
\end{equation}
is continuous at $y=0$ (the jump in the first derivative compensates exactly the change of sign). 
Simplifying the last term by integration by parts on the first component of $j$, we get
\begin{align*}
0 &= j(t,0)  + \frac{\alpha}{2} \int_{-\infty}^\infty   \sign(y)  e^{- \sqrt{\alpha} |y|} \rho(t,y)\, \rd y
+ \frac{\sqrt{\alpha}}{2} \int_{-\infty}^\infty     e^{- \sqrt{\alpha} |y|} \left (\chi \sign(y) + \dxt(t)) \rho(t,y)\right)\, \rd y\,.
\end{align*}
The two middle contributions  vanish because of the conservation law \eqref{dxS}. In addition, by continuity we have 
$$j(t,0) = \frac12\left ( j(t,0^+) + j(t,0^-)\right ) = - \frac12\left (  \partial_y \rho(t,0^+) + \partial_y \rho(t,0^-)\right  )-  \dxt(t) \rho(t,0)\,.$$ 
Therefore, by the representation $S = K_\alpha*\rho$ \eqref{eq:S conv},
\begin{equation*}
0 = \dxt(t) (-\rho(t,0) + \alpha S(t,0)) - \frac12\left (  \partial_y \rho(t,0^+) + \partial_y \rho(t,0^-)\right  )\, ,
\end{equation*}
which is equivalent to \eqref{xdot-main} using \eqref{quasimodel22}.
\end{proof}

\subsection{Reformulation of the problem and conservation laws}\label{sec:CLs}

In fact, using the continuity of the flux $j$ defined in \eqref{flux} at $y=0$, we can reformulate the dynamics of $\bx(t)$ as follows,
\begin{align}\label{eq:xdot3}
\dxt(t)
= \frac{\partial_y\left(e^{\chi|y|}\rho\right)(t,0)}{\partial^2_{y} S(t,0)}\,.
\end{align}
Substituting this expression into \eqref{quasimodelbis}, we find that the dynamics of the cell density $\rho$ are governed by
\begin{align}
 \partial_t \rho(t,y) =  &\partial_{y}\left(e^{-\chi|y|}\partial_y \left(e^{\chi|y|}\rho(t,y)\right) + \left.\left(\frac{\partial_y\left(e^{\chi|y|}\rho(t,y)\right)}{\partial^2_{y}S(t,y)}\right)\right|_{y=0}\rho(t,y) \right)\,.\label{model-yframe-2}
\end{align}

This formulation strongly suggest to work with relative densities: 
\begin{equation}
\label{eq:relative dens}
u(t,y)
:=\frac{\rho(t,y)}{\rho_\infty(y)}
=  e^{\chi|y|}\rho(t,y)\,.
\end{equation}
Then, $u$ satisfies the system
\begin{subequations}
 \label{nonlineq:alphachi}
 \begin{align}
&\partial_t u(t,y) = e^{\chi|y|} \partial_y \left(e^{-\chi|y|} \left(\partial_y u(t,y) + \frac{\partial_y u(t,0)}{ \partial^2_{y} S(t,0)}\, u(t,y)
\right) \right)\, ,\label{nonlineq:alphachi:1}\\
& -   \partial^2_{y}S(t,0)=   u(t,0) -  \frac{\sqrt{\alpha}}{2} \int_{-\infty}^\infty u(t,y) e^{-(\chi+\sqrt{\alpha})|y|} \, \rd y\, .\label{nonlineq:alphachi:3}
\end{align}
\end{subequations}
Note that $\partial_y u$ is continuous at $y=0$, since the $\mathcal C^1$ discontinuity of $\rho$ has been exactly compensated in \eqref{eq:relative dens}. 

From now on, we shall work with \eqref{nonlineq:alphachi}. 
This system admits the stationary state $u_\infty(y)\equiv 1$, and two conservation laws:
\begin{itemize}
\item conservation of mass:
$$
(\forall t\geq 0)\quad \int_{-\infty}^\infty u(t,y) e^{-\chi|y|} \, \rd y = \frac2\chi\,,
$$
(recall that the mass is fixed to $M=2/\chi$), and
\item centering frame:
$$
(\forall t\geq 0)\quad \int_{-\infty}^\infty u(t,y) \sign(y) e^{-(\chi+\sqrt{\alpha})|y|} \, \rd y = 0 
$$
(linked to the invariance by translation in the original problem).
\end{itemize}
In other words, the two conservation laws are given by the following weight functions
$$
\varphi_1(y) = e^{-\chi|y|}, \quad \varphi_2(y) = \sign(y) e^{-(\chi+\sqrt{\alpha})|y|}\,,
$$
such that
$$
\frac{\rd}{\rd t} \int \varphi_i(y) u(t,y)\, \rd y=0\,,\qquad i=1,2\,.
$$
Setting $\lambda:= \chi+\sqrt{\alpha}$ and using the notation $\langle\cdot\rangle_r$ for the weighted average,
$$
\langle  f \rangle_r: = \frac{r}{2} \int  f(y) e^{-r |y|} \, \rd y\, , \qquad f:\R\to\R\,,\quad r>0\,,
$$
the two conservation laws can be written equivalently as
\begin{equation*}
    \langle  u(t) \rangle_\chi= 1\,,\qquad \langle  \partial_y u (t) \rangle_\lambda =0\,, 
    \qquad \text{ for all } \, t\ge 0\,.
\end{equation*}

\section{Energy estimates}

In this section, we derive $H^1$ energy estimates to measure dissipation in \eqref{nonlineq:alphachi}. Interestingly, we find that dissipation always occurs, as the   non-local interaction contribution is overwhelmed by heat dissipation. Our analysis relies on the following Poincar\'e-type inequality which is designed to handle the discrepancy in the exponential rates of the pair of conservation laws (resp. $\chi$ and $\lambda = \chi + \sqrt{\alpha}$). 

\subsection{Improved Poincar\'e inequality}\label{sec:poincare}

\begin{proposition}[Poincar\'e inequality with unusual normalization]\label{lem:poincare}
Assume $\lambda \geq \chi$, then for any $w\in L^1_\lambda$ such that $w' \in L^2_\chi$,
\begin{equation}\label{eq:poincare}
\int_{-\infty}^\infty |w(y) - \langle w \rangle_\lambda|^2 e^{-\chi |y|}\,\rd y \leq \frac 4{\chi^2} \int_{-\infty}^\infty | w'(y)|^2 e^{-\chi|y|}\,\rd y\, .
\end{equation}
Moreover, the constant $\frac 4{\chi^2}$ is optimal.
\end{proposition}

Inequality \eqref{eq:poincare} is the standard Poincar\'e  inequality with  exponential weight when $\lambda = \chi$. In that case, the optimal constant $\frac4{\chi^2}$ can be deduced from spectral analysis of the linear operator $-w'' + \chi(\sign y) w' $ in the Hilbert space $L^2_\chi = L^2  (e^{-\chi|y|}  )$ which admits a spectral gap $(0,\frac {\chi^2}4)$, with an isolated value $0$, and essential spectrum beyond $\frac {\chi^2}4$.

Inequality \eqref{eq:poincare} is an improvement of the standard Poincar\'e inequality ($\lambda = \chi$), simply because 
\begin{equation}\label{eq:classical poincare}
\int_{-\infty}^\infty |w(y) - \langle w \rangle_\chi|^2 e^{-\chi |y|}\,\rd y  = \inf_{m\in \R} \int_{-\infty}^\infty |w(y) - m|^2 e^{-\chi |y|}\,\rd y\, .
\end{equation}
It may appear surprising at first glance that the optimal constant does not depend on the exponent $\lambda \geq \chi$. This is a consequence of the fact that the optimal constant is never reached in \eqref{eq:poincare} because the exponential weight has critical decay, and the putative optimal functions are not in $L^2_\chi$. This yields some room for improving  \eqref{eq:classical poincare} as in \eqref{eq:poincare}.

We are not aware of the occurrence of \eqref{eq:poincare} in the literature. There are some examples of weighted Poincar\'e inequalities with non standard averaging, see for instance \cite[Appendix A.2]{bouin_hypocoercivity_2019}. We can also report the extremal case $\lambda = +\infty$ which coincides with the following Hardy inequality:
\begin{equation}\label{eq:hardy}
\int_{-\infty}^\infty |w(y) - w(0)|^2 e^{-\chi |y|}\,\rd y \leq \frac 4{\chi^2} \int_{-\infty}^\infty |w'(y)|^2 e^{-\chi|y|}\,\rd y\, .
\end{equation}
The link between the latter and the classical Poincar\'e inequality with exponential weights is pointed out in \cite[Corollary 1.3]{miclo_quand_2008}, see also \cite{bobkov_exponential_1999}. Inequality \eqref{eq:poincare} can be viewed as a continuous deformation connecting \eqref{eq:classical poincare} and \eqref{eq:hardy} with a {\em constant} optimal constant. 

\begin{remark}
We recall below the equivalence between \eqref{eq:hardy} and the classical Hardy inequality. We restrict to the half-line $\{y>0\}$ and to $\chi =1$ for the sake of clarity. We make the change of variable $x = k e^y$ on both sides: 
\begin{equation}\label{eq:hardy2}
\int_{k}^\infty |\tilde w(x) - \tilde w(k)|^2 \,\frac{\rd x}{ x^2} \leq 4  \int_{k}^\infty |\tilde w'(x)|^2  \, \rd x\, .
\end{equation}
Letting $k\to 0$ we recover the classical Hardy inequality on the half-line. 
 
\end{remark}

It is noticeable that the optimal constant in \eqref{eq:poincare} does not depend on $\lambda$, provided it is strictly greater than $\chi$. For the proof of Proposition~\ref{lem:poincare} we present two arguments: first some insights based on spectral analysis (not complete enough to provide a rigorous proof), followed by a complete (technical) proof of Proposition~\ref{lem:poincare} based on the reformulation of \eqref{eq:poincare} as a quadratic form. Note that we can always restrict our arguments to $\chi = 1$ without loss of generality by rescaling $y$ to $\chi y$.\\

We begin with an analysis of the spectrum for the operator associated to inequality \eqref{eq:poincare}.
Studying carefully the critical functions $w$ of \eqref{eq:poincare} (such that (i) $\langle w \rangle_\lambda = 0$ and (ii) $\frac12\int w^2e^{-|y|}\, dy = 1$), we find that they must satisfy the following problem:
\begin{equation*}
-w'' +  (\sign y) w' - \mu w = \nu e^{(1 - \lambda)|y|}\, .
\end{equation*}
Here, the scalars $\mu$ and $\nu$ are Lagrange multipliers corresponding to the two constraints (i) and (ii).
Actually, this problem is equivalent to the next one after the change of unknown $w(y) = W(y) e^{|y|/2}$:
\begin{equation}\label{eq:W}
-W'' - \delta_0 W - \left ( \mu - \frac14 \right ) W = \nu e^{(1/2- \lambda)|y|} \,.
\end{equation}
This problem can be explicitly solved after tedious computations by means of the fundamental solution (details not shown). On the one hand, the constant $\nu$ must be adjusted to ensure $\|w\|_{L^2_1} = 1$.  On the other hand, the condition $\langle w \rangle_\lambda = 0$ cannot be satisfied if $\mu < \frac{1}4$. When $\mu =  \frac{1}4$, we have the following expression (general solution of the homogeneous problem augmented by a particular solution):
\begin{equation}
w =  C \left ( \frac12 |y| - 1 \right )  e^{|y|/2} 
- \frac{\nu}{\left (\lambda - \frac12\right )^2} \left (  
   (\lambda - 1)     |y|e^{|y|/2}    +   e^{(1 - \lambda)|y|}\right )\, ,
\end{equation} 
where  the constant $C$ can be adjusted to ensure $\langle w \rangle_\lambda = 0$ when $\lambda >1$.
  
However, one immediately notices that $w\notin L^2_1$, simply because $\frac{1}4$ belongs to the essential spectrum of $-w'' + (\sign y) w' $. Nonetheless, considering the quotient 
\begin{equation}
\dfrac{\int_{-R}^R | w'(y)|^2 e^{-|y|}\,\rd y}{\int_{-R}^R |w(y)|^2 e^{- |y|}\,\rd y}
\end{equation}
for arbitrary large values of $R$, we see that the additional contribution of the particular solution $\mathcal O\left (e^{(1 - \lambda)|y|}\right )$ becomes negligible, so that the quotient converges to the constant $\frac{1}{4}$ as   $R\to +\infty$, independently of $\lambda$. This is the reason why the optimal constant in \eqref{eq:poincare} does not depend on $\lambda$ in the case $\lambda>1$.\\

Next, we provide a complete proof of Proposition~\ref{lem:poincare} by direct computations. It relies on the following technical lemma, whose proof is postponed to
Appendix~\ref{sec:poincareproof}.

\begin{lemma}\label{lem:Omega}
For $\lambda\ge \chi>0$, define $\Omega_{\lambda,\chi}:\R\times\R\to\R$ by 
\begin{equation}\label{def:Omega}
\Omega_{\lambda,\chi}(x,y) =
\begin{cases}
\left(M_\lambda(x) - M_\chi(x)\right)\left ( M_\lambda(y)  - M_\chi(y)\right ) + (1 - M_\chi(y))  M_\chi(x) & \quad \text{if $x\le y$}\,,\\
\left ( M_\lambda(x) - M_\chi(x)\right )\left ( M_\lambda(y)  - M_\chi(y)\right ) + (1 - M_\chi(x))  M_\chi(y) & \quad \text{if $x>y$}\,,
\end{cases}
\end{equation}
where $M_\lambda$ denotes the cumulative density function, $$M_\lambda(x) = \int_{z < x} \frac\lambda2 e^{-\lambda |z|} \,\rd z \,.$$
Then $\Omega_{\lambda,\chi}$ is non-negative and symmetric,
$$
\Omega_{\lambda,\chi}\ge 0\,,\qquad \Omega_{\lambda,\chi}(x,y)=\Omega_{\lambda,\chi}(y,x)\,,
$$
and we can rewrite the left-hand side of the Poincar\'e inequality \eqref{eq:poincare} as
\begin{equation*}
\frac12\int_{-\infty}^\infty |w(y) - \langle w \rangle_\lambda|^2 e^{- \chi |y|}\,\rd y 
 =   \iint   w'(x_1)   w'(x_2) \Omega_{\lambda,\chi}(x_1,x_2 )   \,\rd x_1 \,\rd x_2   \,.
\end{equation*}
\end{lemma}

\begin{proof}[Proof of Proposition~\ref{lem:poincare}]
We assume again $\chi=1$, denote $\Omega_{\lambda}:=\Omega_{\lambda,1}$, and introduce the new function $W(x) =  w'(x) e^{-|x|/2}$. It then follows from Lemma~\ref{lem:Omega} that the Poincar\'e inequality \eqref{eq:poincare} is equivalent to
\begin{equation}\label{eq:quad32}
 \iint W(x_1)  W(x_2) \Omega_\lambda(x_1,x_2 )  e^{|x_1|/2}e^{|x_2|/2}   \,\rd x_1 \,\rd x_2 \leq 2 \int |W(x)|^2 \, \rd x\, . 
\end{equation}
The quadratic form on the left hand side can be reformulated as follows:
\begin{align*}
& \iint W(x_1)  W(x_2) \Omega_\lambda(x_1,x_2 )  e^{|x_1|/2}e^{|x_2|/2}   \,\rd x_1 \,\rd x_2 \\
&   =
\frac12\iint \left [ W(x_1) - W(x_2) \right ] \left [  W(x_2) - W(x_1)\right ] \Omega_\lambda(x_1,x_2 )  e^{|x_1|/2}e^{|x_2|/2}   \,\rd x_1 \,\rd x_2 \\
 & \quad  + \iint |W(x)|^2 \Omega_\lambda(x,y )  e^{|x|/2}e^{|y|/2}   \,\rd x \,\rd y \,.
\end{align*}
Therefore, in regard to \eqref{eq:quad32} and thanks to non-negativity of $\Omega_\lambda$ provided by Lemma~\ref{lem:Omega}, it is sufficient to prove the pointwise inequality
\begin{equation}\label{eq:pointwise}
 (\forall x)\quad e^{|x|/2} \int  \Omega_\lambda(x,y )  e^{|y|/2}    \,\rd y  \leq  2\, .
 \end{equation} 
 In fact, the left-hand-side above is independent of $\lambda$. To see this, note that $M_\lambda(-y)=1-M_\lambda(y)$, and so for any $R>0$,
 \begin{align*}
     \int_{-R}^R M_\lambda(y) e^{|y|/2}\, \rd y 
     &= \int_0^R \left(1-M_\lambda(y)\right) e^{|y|/2}\,\rd y + \int_0^R M_\lambda(y)e^{|y|/2}\\
     &= \int_0^R e^{|y|/2}\,\rd y = \frac{1}{2} \left(e^{R/2}-1\right)\,.
 \end{align*}
 Since this expression is independent of $\lambda$ for all $R>0$, we obtain
\begin{equation*}
\int \left ( M_\lambda(y)  - M_1(y)\right ) e^{|y|/2}    \,\rd y  = 0\, .
\end{equation*} 
Hence, it is enough to check the inequality \eqref{eq:pointwise} for $\lambda = 1$, {\em i.e.}
\begin{equation}\label{eq:bound}
 (\forall x)\quad e^{|x|/2} \int  \left  ( M_1(\min\{x,y\}) - M_1(x) M_1(y)  \right )  e^{|y|/2}    \,\rd  y  \leq  2\, .
\end{equation}
For $x>0$, this corresponds to showing
\begin{align}\label{eq:xposbound}
  (\forall x)\quad  &e^{x/2} \left (  \left  ( 1  - M_1(x)  \right )  \int_{-\infty}^x  M_1(y) e^{|y|/2}    \,\rd  y 
 +  M_1(x) \int_x^{\infty} \left  ( 1  - M_1(y)  \right )  e^{y/2}    \,\rd  y  \right ) \le 2\,.
\end{align}
Fix $R>x>0$. Again using $M_1(-y)=1-M_1(y)$, we have
\begin{align*}
    \int_{-R}^x M_1(y)e^{|y|/2}\, \rd y
    &= \int_0^R (1-M_1(y))e^{y/2}\,\rd y + \int_0^x M_1(y)e^{y/2}\,\rd y\\
    &= \int_0^R e^{y/2}\,\rd y - \int_x^R M_1(y)e^{y/2}\,\rd y\,.
\end{align*}
Hence, 
\begin{align*}
     &\left  ( 1  - M_1(x)  \right )  \int_{-R}^x  M_1(y) e^{|y|/2}    \,\rd  y 
 +  M_1(x) \int_x^{R} \left  ( 1  - M_1(y)  \right )  e^{y/2}    \,\rd  y \\
 &\quad =\left  ( 1  - M_1(x)  \right ) \left(\int_0^R e^{y/2}\,\rd y - \int_x^R M_1(y)e^{y/2}\,\rd y\right)
 +  M_1(x) \int_x^{R} e^{y/2}    \,\rd  y
 - M_1(x) \int_x^{R} M_1(y)    e^{y/2}    \,\rd  y\\
  &\quad =\int_0^R e^{y/2}\,\rd y 
  -\int_x^R M_1(y)e^{y/2}\,\rd y
  -M_1(x)\int_0^x e^{y/2}\,\rd y\\
  &\quad = -2+2e^{x/2}-\left(e^{-R/2}-e^{-x/2}\right) -\left(1-\frac{1}{2}e^{-x}\right)2\left(e^{x/2}-1\right)\\
  &\quad =-e^{-R/2} +e^{-x/2}\left(2-e^{-x/2}\right)\,.
\end{align*}
Taking $R\to\infty$, and multiplying by $e^{x/2}$ yields the desired bound \eqref{eq:xposbound}. Note that equality is achieved only in the limit $x\to\infty$. 
Similarly, simplifying \eqref{eq:bound} for $x<0$ and replacing $x$ with $-x$, one obtains again \eqref{eq:xposbound}.
This concludes the proof of the Poincar\'e inequality \eqref{eq:poincare}.
\end{proof}

\subsection{Energy dissipation}

To analyse the stability of solutions to \eqref{nonlineq:alphachi}, we define $v$ as the perturbation around the equilibrium, 
$$u=u_\infty+v=1+v\,,$$
Denoting $w(t,y):=\partial_y v(t,y)$, the perturbation $v$ satisfies
\begin{subequations}
 \label{model-v}
 \begin{align}
&\partial_t v(t,y) = e^{\chi|y|}\partial_y\left(e^{-\chi|y|} \left(\partial_yv(t,y)- \mu[v(t)] (1+v(t,y))\right)\right)\, ,\label{model-v:1}\\
& \mu[v(t)]:=\frac{\lambda w(t,0)}{\chi+\lambda v(t,0)-\sqrt{\alpha} \langle v(t) \rangle_\lambda}\, .\label{model-v:2}
\end{align}
\end{subequations}
and the two conservation laws rewrite as
$$\langle v(t) \rangle_\chi=0\, , \qquad \langle w(t)\rangle_\lambda=0\,,\qquad \forall t\ge 0\, . $$ 
 
Next, we introduce the weigthed energy functionals
\begin{align*}
 &E(t)=\dfrac12 \int_{-\infty}^\infty \left|v(t,y)  \right|^2 e^{-\chi |y|}\,\rd y\,,\\
 &F(t)=\dfrac12 \int_{-\infty}^\infty |w(t,y)|^2 e^{-\chi |y|}\,\rd y\,,\\
 &G(t)=\dfrac12 \int_{-\infty}^\infty |\partial_y w(t,y)|^2 e^{-\chi|y|}\,\rd y\,.
\end{align*}
In short, we denote $E(t) = \frac{1}{2}\|v(t)\|^2_{L^2_\chi}$, $F(t) =\frac12 \|w(t)\|^2_{L^2_\chi}$, and $G(t) = \frac{1}{2}\| \partial_y w (t)\|^2_{L^2_\chi}$, where $\|\cdot\|_{L^2_\chi}$ denotes the weighted $L^2$ norm.

The weighted $H^2$ norm $G$ is introduced in order to control the pointwise term $w(t,0)$ in \eqref{model-v:2} that cannot be controlled in $H^1$ regularity. Our main estimate is the following.

 \begin{proposition}[Entropy dissipation]\label{prop:Fdiss}
 The dissipation of $F$ along solutions $v(t,y)$ to equation \eqref{model-v} is given by
\begin{equation}
 \frac{\rd}{\rd t} F(t)=-2G(t) + 2\sqrt{\alpha} w(t,0)^2  + J(t)\,,\label{dtF alpha}
\end{equation}
where $J$ is a higher-order (genuinely non-linear) contribution defined in \eqref{nonlinJ}.
\end{proposition}

\begin{proof}
Differentiating $F(t)$ along solutions of \eqref{model-v}, substituting the evolution of $\partial_t w$,
and integrating by parts, we have
\begin{align*}
\dot{F}
&= \int_{-\infty}^\infty w\partial_t w e^{-\chi|y|}\, \rd y
= \int_{-\infty}^\infty w\partial_y\left(e^{\chi|y|}\partial_y\left(e^{-\chi|y|} \left(w-(1+v)\mu[v]\right)\right)\right)e^{-\chi|y|}\, \rd y\\
&= -\int_{-\infty}^\infty e^{\chi|y|} \left |\partial_y \left (  e^{-\chi|y|} w \right ) \right |^2 \, \rd y
 + \mu[v] \int_{-\infty}^\infty e^{\chi|y|} \partial_y \left (  w e^{-\chi|y|}\right )  \partial_y\left(e^{-\chi|y|} (1+v) \right) \, \rd y
\\
&= I_1+I_2\,.
\end{align*}
We evaluate the terms $I_1$ and $I_2$ separately. Expanding $I_1$, we find
\begin{align*}
I_1
&=-\int_{-\infty}^\infty  e^{-\chi|y|} \left | \partial_y w - \chi (\sign y) w \right |^2 \, \rd y\\
&= -\int_{-\infty}^\infty e^{-\chi|y|} \left | \partial_y w  \right |^2 \, \rd y + 2 \chi \int_{-\infty}^\infty e^{-\chi|y|}   \partial_y w  (\sign y) w \, \rd y
- \chi^2 \int_{-\infty}^\infty e^{-\chi|y|} w^2\, \rd y
\\
&=  -\int_{-\infty}^\infty e^{-\chi|y|} \left | \partial_y w  \right |^2 \, \rd y
- \chi \int_{-\infty}^\infty \partial_y \left ( e^{-\chi|y|} (\sign y) \right ) |w|^2 \, \rd y  
- \chi^2 \int_{-\infty}^\infty e^{-\chi|y|} |w|^2\, \rd y
\\
&= -\int_{-\infty}^\infty e^{-\chi|y|} \left | \partial_y w  \right |^2 \, \rd y - 2 \chi w(t,0)^2 \,.
\end{align*}
Dealing with $I_2$, we separate the quadratic terms from the higher order contributions, 
\begin{align*}
I_2
&= \frac\lambda\chi w(t,0)\int_{-\infty}^\infty e^{\chi|y|} \partial_y \left (  w e^{-\chi|y|}\right )  \partial_y\left(e^{-\chi|y|}  \right) \, \rd y
\\
& \quad + \frac\lambda\chi w(t,0) \int_{-\infty}^\infty e^{\chi|y|} \partial_y \left (  w e^{-\chi|y|}\right )  \partial_y\left(e^{-\chi|y|} v  \right) \, \rd y
\\
& \quad + \left ( \mu[v] -  \frac\lambda\chi w(t,0) \right ) \int_{-\infty}^\infty e^{\chi|y|} \partial_y \left (  w e^{-\chi|y|}\right )  \partial_y\left(e^{-\chi|y|} (1+v) \right) \, \rd y\\
& = I_2' + J_1 + J_2\, . 
\end{align*}
We can rearrange the quadratic term $I_2'$ as follows
\begin{align*}
I_2' &= -  \lambda w(t,0)\int_{-\infty}^\infty  \partial_y \left (  w e^{-\chi|y|}\right ) (\sign y) \, \rd y \\
& = 2 \lambda w(t,0)^2 \, .
\end{align*}
This concludes the proof, with $J(t):=J_1(t)+J_2(t)$.
\end{proof}

In order to turn the entropy dissipation result Proposition~\ref{prop:Fdiss} into exponential relaxation to equilibrium (see Section~\ref{sec:cv}), we will make use of the improved Poincar\'e inequality \eqref{eq:poincare} together with the following lemma: 
\begin{lemma}[Interpolation inequality]\label{lem:w0}
Fix $a\ge b>0$. For any function $f\in L^1_a(\R)$ such that $f'\in L^2_b(\R)$, we have 
\begin{equation}
\left | f(0) -  \langle f \rangle_a \right|^2 \leq  \left(\frac1{2a-b}\right) \frac12 \int |f'(y)|^2e^{-b|y|}\,\rd y\, .
\end{equation}
\end{lemma}

\begin{proof}
We first perform some singular integration by parts in order to relate $f(0)$ and $f'(y)$. 
\begin{equation*}
\frac1 2 \int_\R f'(y) (\sign y) e^{-a |y|}\, \rd y = \frac{a}{2}\int_\R f  e^{-a |y|}\, \rd y - f(0) = \langle f \rangle_a -f(0)\, . 
\end{equation*}
By  Young's inequality, we find
\begin{align*}
 |f(0) - \langle f \rangle_a| &= \left|\frac{1}{2} \int_\R \sign(y) f'(y) e^{-a|y|}\, \rd y \right| 
 \leq \frac{1}{2} \int_\R |f'(y)| e^{-a|y|} \, \rd y \\
 &\leq \frac{1}{2} \left(\int_\R |f'(y)|^2 e^{-b|y|} \, \rd y \right)^{1/2}\left(\int_\R e^{-(2a-b)|y|} \, \rd y \right)^{1/2}\\
 & \leq  \frac12 \left ( \dfrac{2}{2a-b} \right )^{1/2} \left(\int_\R |f'(y)|^2 e^{-b|y|} \, \rd y \right)^{1/2}\,.
\end{align*}
Raising it to the square, we obtain the result. 
\end{proof}

\section{Exponential relaxation to equilibrium (perturbative analysis)}\label{sec:cv}

This section is devoted to the proof of Theorem \ref{thm:main}, with the help of the previous energy estimates. We proceed in two steps. Firstly, we prove that the shape converges to equilibrium (meaning that $u$ converges to a constant unit value). Secondly, we prove that the center $\bx(t)$ converges to a finite value (not determined).

\begin{proof}[Proof of Theorem \ref{thm:main} -- Step 1: convergence of the shape]
Applying Lemma \ref{lem:w0} to $w\in L^1_\lambda$ with $w'\in L^2_\chi$, and noting that $\langle w \rangle_\lambda=0$ by the second conservation law, we have
\begin{equation}\label{w0est}
     \left | w(t,0)  \right|^2 \leq  \frac 1{\chi + 2\sqrt{\alpha}}\,G(t)\, .
\end{equation}
Combining with Proposition \ref{prop:Fdiss}, we find
\begin{equation}
 \frac{\rd}{\rd t} F(t)\le-\left ( \frac{\chi + \sqrt{\alpha}}{\frac\chi2 + \sqrt{\alpha}} \right ) G(t) +   J(t)\,,\label{dtF alpha 2}
 \end{equation} 
where the higher-order contributions in $J(t)$ reads as follows,
\begin{multline}\label{nonlinJ}
J(t) = \frac\lambda\chi w(t,0) \int_{-\infty}^\infty e^{\chi|y|} \partial_y \left (  w e^{-\chi|y|}\right )  \partial_y\left(e^{-\chi|y|} v  \right) \, \rd y
\\
 + \frac{\lambda}{\chi} \left( \dfrac{w(t,0)\left ( \sqrt{\alpha} \langle v(t) \rangle_\lambda - \lambda v(t,0) \right ) }{\chi+\lambda v(t,0)-\sqrt{\alpha} \langle v(t) \rangle_\lambda} \right) \int_{-\infty}^\infty e^{\chi|y|} \partial_y \left (  w e^{-\chi|y|}\right )  \partial_y\left(e^{-\chi|y|} (1+v) \right) \, \rd y\, .
\end{multline}
The first integral involves $\| \partial_y w \partial_y v \|_{L^1_\chi}$,  $\| \partial_y w   v \|_{L^1_\chi}$,  $\|  w \partial_y v \|_{L^1_\chi}$ and  $\| w v \|_{L^1_\chi}$, which are all controlled quadratically by $E,F$ and $G$. In addition, the prefactor $w(t,0)$ is controlled by $G^{1/2}$ following equation \eqref{w0est}.  Similarly, the second contribution in \eqref{nonlinJ} follows the same pattern, as both $ \langle v(t) \rangle_\lambda$ and $v(t,0)$ are controlled by interpolation using Lemma~\ref{lem:w0}:
\begin{align*}
&\left | v(t,0) -  \langle v(t) \rangle_\lambda \right|^2 \leq  \frac 1{\chi + 2\sqrt{\alpha}}F(t)\,,\qquad
\left | v(t,0) -  \langle v(t) \rangle_\chi \right|^2 \leq  \frac 1{\chi}F(t)\, ,
\end{align*}
and the conservation law $ \langle v(t) \rangle_\chi = 0$. It follows that
\begin{align*}
    \chi+\lambda v(t,0)-\sqrt{\alpha} \langle v(t) \rangle_\lambda 
    \ge \chi-\left( \sqrt{\chi}+\sqrt{\frac{\alpha}{\chi+2\sqrt{\alpha}}}\right) F^{1/2}(t)\,,
\end{align*}
and so
\begin{align*}
    \dfrac{\sqrt{\alpha} \langle v(t) \rangle_\lambda - \lambda v(t,0) }{\chi+\lambda v(t,0)-\sqrt{\alpha} \langle v(t) \rangle_\lambda} 
    \le C \frac{F^{1/2}(t)}{1-CF^{1/2}(t)}\,,
\end{align*}
for some constant $C>0$ depending on $\chi$ and $\alpha$, provided that $F$ is small enough.
Therefore, 
\begin{align}\label{eq:ninlin}
|J(t)| &\leq C G(t)^{1/2} \left ( G(t)^{1/2} F(t)^{1/2} + G(t)^{1/2} E(t)^{1/2}  + F(t) + F(t)^{1/2} E(t)^{1/2}  \right )
\\
&\qquad +C G(t)^{1/2} \left ( \frac{ F(t)^{1/2}}{1 - C  F(t)^{1/2} }\right ) \times \notag\\
&\qquad \quad \left ( G(t)^{1/2} + F(t)^{1/2} +G(t)^{1/2} F(t)^{1/2}  + G(t)^{1/2} E(t)^{1/2}  + F(t) + F(t)^{1/2} E(t)^{1/2}  \right )\notag
\end{align}
for some constant $C$ depending on $\chi$ and $\alpha$, provided that $F$ is small enough. Now, by the improved Poincar\'e inequality Proposition~\ref{lem:poincare} and the second conservation law, we obtain $$F(t) \leq \frac4{\chi^2} G(t)\,.$$  Moreover, the classical Poincar\'e inequality with exponential weight together with the first conservation law yield $$E(t) \leq \frac4{\chi^2} F(t)\,.$$
The important point is to isolate the terms with higher derivatives, that is, the terms that are controlled by $G(t)$. In fact,  using the two inequalities above, the bound \eqref{eq:ninlin} can be simplified to the following estimate:
\begin{equation}\label{eq:est J}
|J(t)| \leq C \left ( \dfrac{ F(t)^{1/2}  +  F(t) }{1 - C F(t)^{1/2}} \right ) G(t)\,.
\end{equation}
As a consequence, we can control the relaxation of $F(t)$ using the estimate \eqref{dtF alpha 2}: assuming that $F(0)$ is small enough, 
we get 
\begin{equation}
 \frac{\rd}{\rd t} F(t)\le - \left (  2\gamma_0 - C F(t)^{1/2} \right ) F(t) \,, \quad \text{with}\quad \gamma_0 = \frac{\chi^2}{8} \left ( \frac{\chi + \sqrt{\alpha}}{\frac\chi2 + \sqrt{\alpha}} \right )\,. \label{dtF alpha 3}
\end{equation} 
We conclude that $F(t)$, hence $E(t)$, relaxes exponentially fast to zero, provided that $F(0)$ is initially small enough. Moreover, the rate of convergence is asymptotically close to $\gamma_0$ as given in \eqref{sigma}. 
To conclude the proof of the exponential decay \eqref{eq:exp relax}, we simply note that
\begin{equation*}
\|\tilde \rho(t,\cdot) - \tilde \rho_\infty\|_{H^1_\chi}^2 = 2E(t)+2F(t) \le \left(\frac{8}{\chi^2}+2\right)F(t)\,.
\end{equation*}
\end{proof}

\begin{remark}
It follows from \eqref{dtF alpha 3} that the asymptotic rate of convergence $\gamma$ can be chosen arbitrarily close to the upper bound $\gamma_0$, however at the expense of increasing the prefactor as mentioned in Theorem~\ref{thm:main}. Additionally, it also means that in Theorem~\ref{thm:main} we can choose any $\eps_0$ that satisfies
$
0<\eps_0< 4\gamma_0^2/C^2
$
where $C>0$ as defined in \eqref{dtF alpha 3}.
More precisely, if initially $\|\tilde \rho_0 - \tilde \rho_\infty\|_{H^1_\chi}^2$ is close to the upper bound, i.e. $F(0)=\frac{4}{C^2}(\gamma_0-\delta)^2$ with $\delta>0$ arbitrarily small, then the above argument still provides decay to equilibrium, with a potentially very slow rate of convergence $0<\gamma'\le \delta$. It follows that for any arbitrarily small $0<\eps<\eps_0$, there exists a large enough time $T_\eps>0$ such that $F(T_\eps)\le \eps$. From \eqref{dtF alpha 3} we conclude
$$
F(t)\le \left( \eps e^{2\gamma T_\eps}\right) e^{-2\gamma t}\qquad \text{ for all } \, t>T_\eps\,,
$$
that is, the prefactor becomes large as  $T_\eps\to\infty$ together with $\gamma\to\gamma_0$.
\end{remark}

\begin{proof}[Proof of Theorem \ref{thm:main} -- Step 2: convergence of the center]
We recall the dynamics \eqref{eq:xdot3}--\eqref{eq:relative dens} of the center $\bx(t)$:
\begin{align*}
    \dxt(t)
    =\frac{w(t,0)}{\partial^2_{y}S(t,0)}\,.
\end{align*}
On the one hand, we have \eqref{nonlineq:alphachi:3}:
\begin{equation}
- \partial^2_{y}S(t,0) = \frac{1}{\lambda} \left ( \chi +  \lambda v(t,0) -  \sqrt{\alpha} \langle v(t) \rangle_\lambda \right ) \, . 
\end{equation}
By the same argument as above, $-\partial^2_{y}S(t,0)$ is bounded below uniformly for $t>0$ provided that $F(0)$ is small enough, simply because the term $\lambda v(t,0) -  \sqrt{\alpha} \langle v(t) \rangle_\lambda$ is controlled by $F(t)$ which is exponentially decaying in that case.  

On the other hand, the value $w(t,0)$ is estimated by the following basic inequality:
\begin{align}\label{eq:w0est-basic}
    |w(t,0)|
    \le C\left ( G(t) + F(t)^{1/3} + F(t)^{1/2} \right )\,.
\end{align}
Indeed, we have by integration by parts,
\begin{align*}
    -\frac{1}{2}\int_{-\infty}^{\infty} \partial_y\left(w^2\right)\sign(y)e^{-\chi|y|}\, \rd y
    =w(0)^2- \frac{\chi}{2}\int_{-\infty}^{\infty} w^2 e^{-\chi|y|}\, \rd y
    =w(0)^2-\chi F(t)\,.
\end{align*}
It follows that
\begin{align*}
    |w(t,0)|^2
    &=\left| -\frac{1}{2}\int_{-\infty}^{\infty} \partial_y\left(w^2\right)\sign(y)e^{-\chi|y|}\, \rd y
 +\chi F(t)\right|\\
      &\le \left(\int_{-\infty}^{\infty} \left|w \right|^2e^{-\chi|y|}\, \rd y\right)^{1/2}\left( \int_{-\infty}^{\infty} \left|\partial_y w\right|^2e^{-\chi|y|}\, \rd y\right)^{1/2}
 +\chi F(t)\\
       &\le 2F(t)^{1/2}G(t)^{1/2}
 +\chi F(t)\,.
\end{align*}
Therefore,
\begin{align*}
    |w(t,0)|
       &\le \left(2F(t)^{1/2}G(t)^{1/2}
 +\chi F(t)\right)^{1/2}\\
       &\le \sqrt{2}F(t)^{1/4}G(t)^{1/4}
 +\sqrt{\chi} F(t)^{1/2}\\ 
        &\le 
C \left ( G(t)+ F(t)^{1/3}+  F(t)^{1/2}\right ) \,,
\end{align*}
where the last estimate follows from Young's inequality. This concludes the proof of the bound \eqref{eq:w0est-basic}.

We deduce immediately that $\dxt$ is integrable, since $F$ is exponentially decaying, whereas $G$ is part of the dissipation of $\dot F$ by \eqref{dtF alpha 2} and \eqref{eq:est J}. 
\end{proof}


\section{Numerics}

\begin{figure}
\begin{center}
\includegraphics[width = 0.45\linewidth]{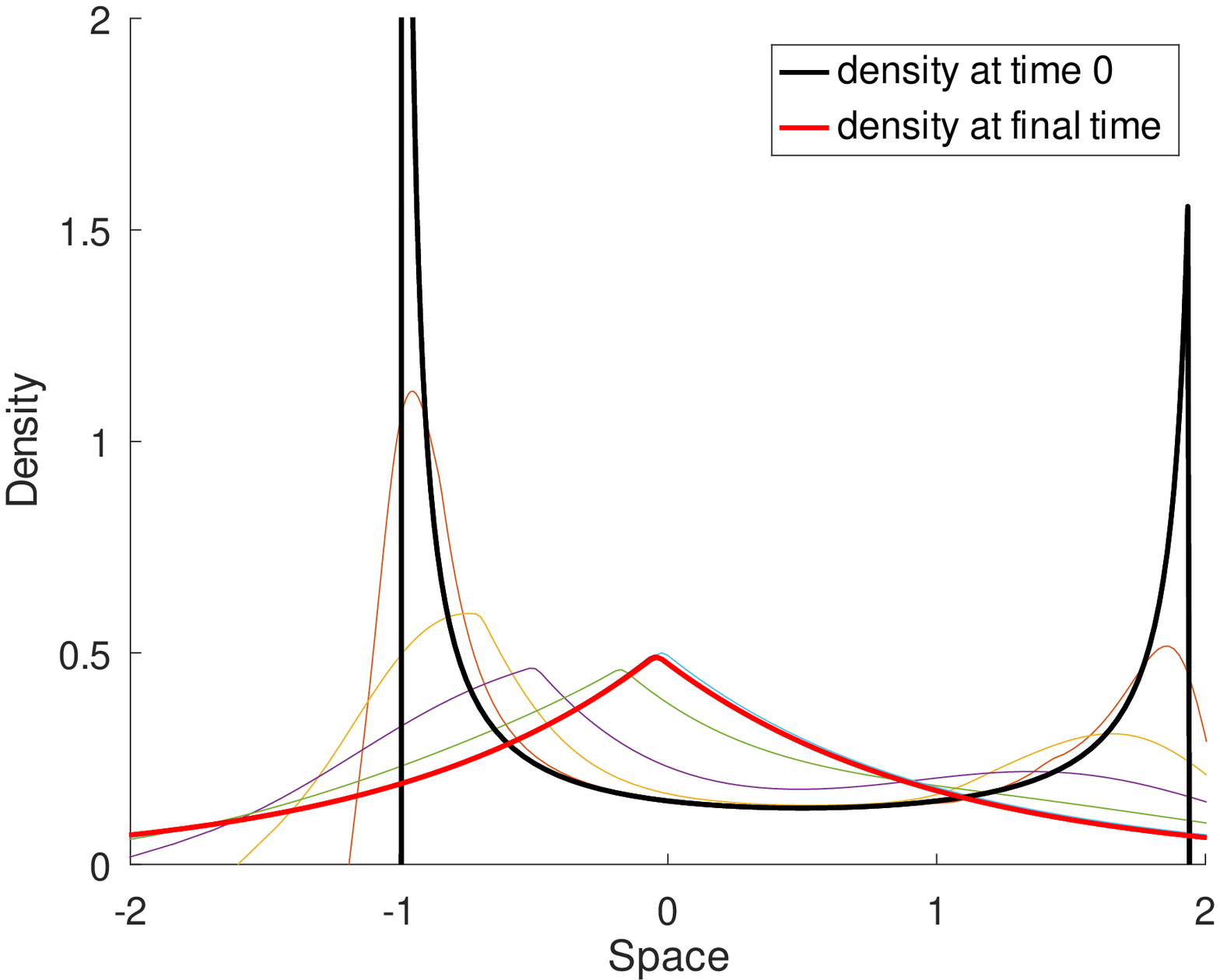}(a)\;
\includegraphics[width = 0.45\linewidth]{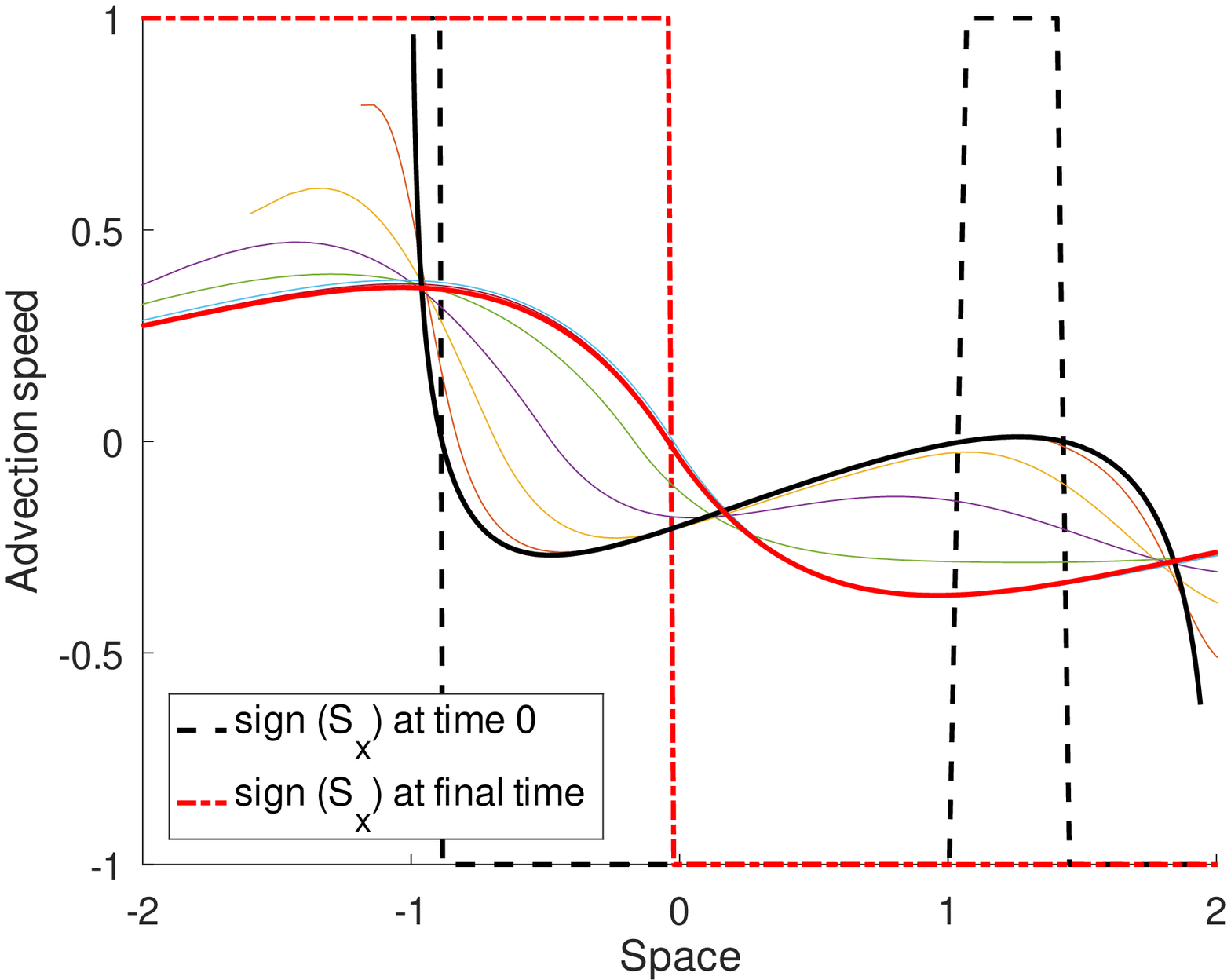}(b)\\
\includegraphics[width = 0.45\linewidth]{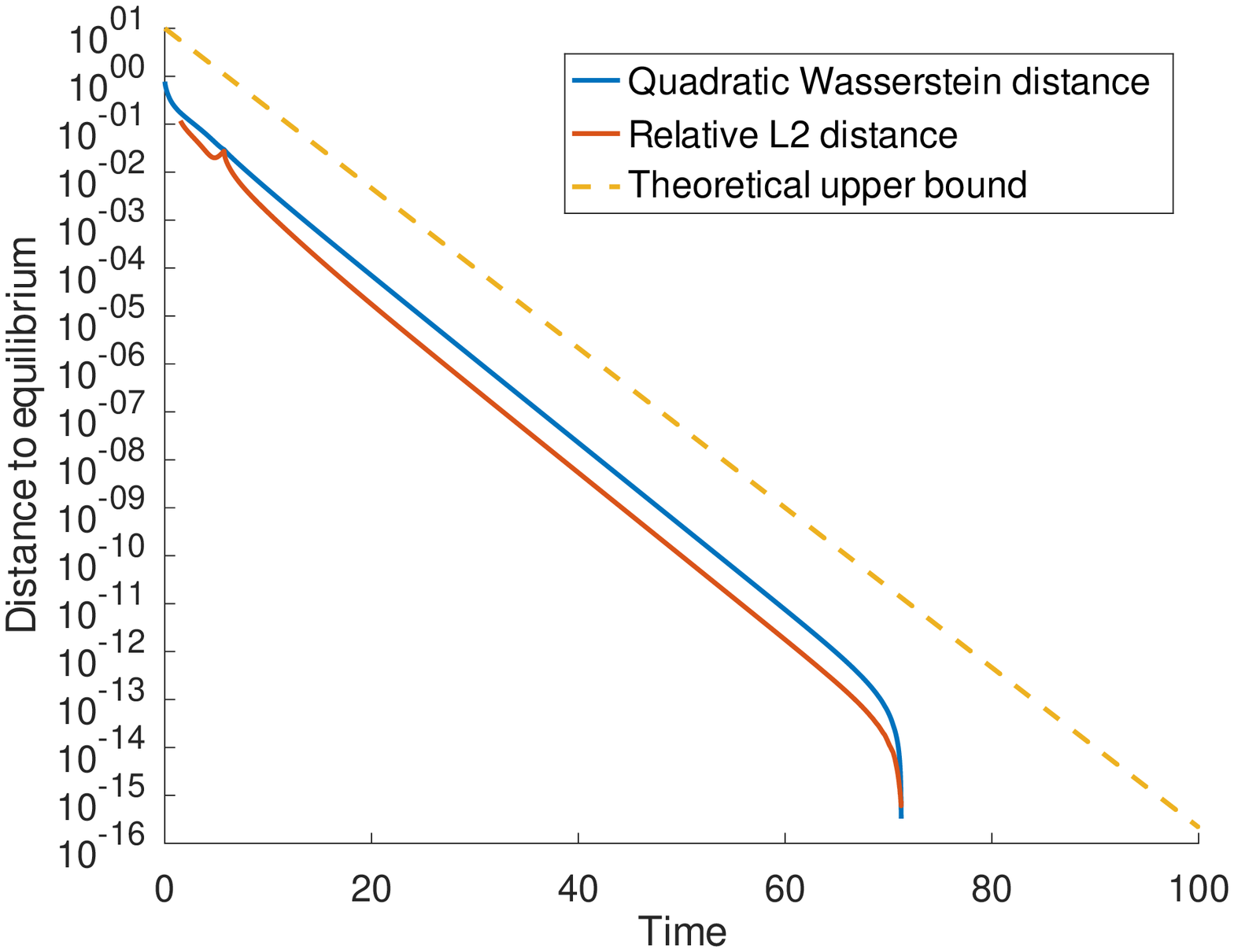}(c)\;
\includegraphics[width = 0.45\linewidth]{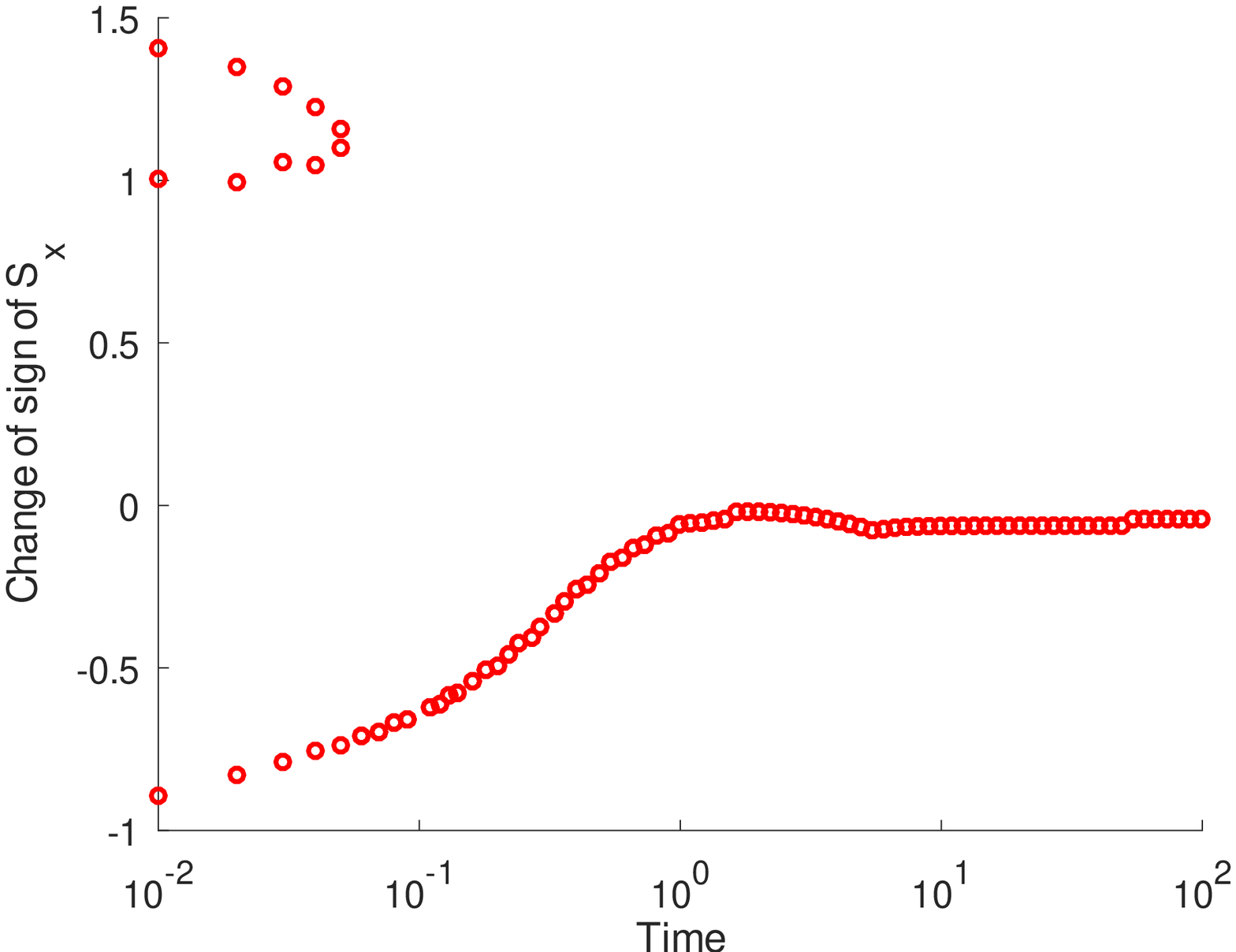}(d)\\
\caption{\small Convergence to equilibrium in system (\ref{quasimodel1})-(\ref{quasimodel2}): (a)  Convergence of the density $\rho$ towards the equilibrium state \eqref{eq:statstate1}. Despite having two peaks initially, it converges towards the stationary profile \eqref{eq:statstate1}. (b) Time evolution of the gradient $\partial_x S$. The initial configuration $S_0(x)$ has multiple critical points, violating the assumption in Proposition \ref{prop:H1}. However, the convergence still holds, as confirmed by (c), with an actual convergence rate better than the theoretical upper bound $\mathcal O(e^{-\gamma_0 t})$. (d) Location of the critical points of $S$ (change of sign of $\partial_x S$) as a function of time: there are three critical points initially, soon merging into a single $\bx(t)$ as assumed in Proposition \ref{prop:H1}. Parameters are: $\chi = \alpha = 1$. Discretization parameters are: $ \Delta \eta = \Delta t = 10^{-2}$.}
\label{fig}
\end{center}
\end{figure}

We performed a numerical investigation of  system (\ref{quasimodel1})-(\ref{quasimodel2}). We used a Lagrangian formulation, based on the inverse of the cumulative distribution function of $\rho$, that is, we define $\Pi$ and its inverse $X$ such that:
\begin{align*}
 &\Pi(t,x)=\int_{-\infty}^x \rho(t,y)\,dy,\\
\Pi(t,X(t,&\eta))=\eta, \qquad
X(t,\Pi(t,x))=x.
\end{align*}
By using the formulation of $S$ as a convolution with the fundamental solution of the elliptic problem (\ref{quasimodel2}), we obtain the following closed equation on $X(t,\eta)$:
\begin{equation*}
 \partial_t X(t,\eta)=-\partial_\eta \left( \frac{1}{\partial_\eta X(t,\eta)}\right) 
 - \chi \sign\left( \frac{1}{2} \int_0^1 \sign\left(X(t,\eta)-X(t,\tilde \eta)\right) e^{-\sqrt{\alpha}\left|X(t,\eta)-X(t,\tilde \eta)\right|}\,d \tilde \eta\right).
\end{equation*}
This formulation is well suited for the numerical simulation of interacting particles in one dimension of space, see  {\em e.g.} \cite{MR2255453,MR3431250}. We discretize the cumulative mass $\eta_i = i \Delta \eta$, for some step mass $\Delta \eta >0$. The discretization in mass rather than space provides more accuracy in regions of higher density. 
  
We use an implicit Euler scheme for the diffusion part, and an explicit scheme for the interaction part of the equation. Being given the solution $X_t(i)$ at time $t$ and at cumulative mass $\eta=i \Delta \eta$, we compute $X_{t+\Delta t}(i)$ by solving
\begin{multline*}
0 = X_{t+\Delta t}(i)- X_t(i)
+ \frac{\Delta t}{\Delta \eta} \,\left( \frac{\Delta \eta}{\left(X_{t+\Delta t}(i+1)-X_{t+\Delta t}(i)\right)} -  \frac{\Delta \eta}{\left(X_{t+\Delta t}(i)-X_{t+\Delta t}(i-1)\right)}\right) \\
+ \chi\, \Delta t \,  \sign\left( \frac{1}{2} \sum_j  \sign\left(X_t(i)-X_t(j)\right) e^{-\sqrt{\alpha}\left|X_t(i)-X_t(j)\right|}\right)
\end{multline*}
using the \copyright Octave  function  \emph{fsolve}. Then  the associated discretized density $\rho$ is recovered from $X$ by doing the reverse transformation using $\rho(t,X_t(i))=2\Delta \eta / \left(X_t(i+1)-X_t(i-1)\right)$.

Typical results are shown in Figure \ref{fig} for an asymmetrical initial density $\rho_0$ which is a relatively large perturbation of the equilibrium profile, associated with a initial signal $S_0$ that does not satisfy the condition of having a unique critical point (as in Proposition \ref{prop:H1}), see Figure \ref{fig}(d). However, after some transient period, it does admit a unique critical point. This suggests that the convergence holds beyond a perturbative regime as analyzed in the present work. We can also see from the numerical results that the theoretical rate of convergence is slightly underestimated.

We can draw a couple of perspectives from this numerical study. (i) It would be interesting to prove that no more than a critical point of $S$ can persist in the long time asymptotics, so that any initial configuration falls into the scope of Proposition \ref{prop:H1} after some time, as in Figure \ref{fig}. (ii) The rate of convergence \eqref{sigma} can certainly be improved with a more cautious analysis, and the identification of a suitable functional inequality. 


\section*{Acknowledgments}
This project has received funding from the European Research Council (ERC) under the European Union’s Horizon 2020 research and innovation programme (grant agreement No 639638).
FH was partially supported by the von Karman postdoctoral instructorship at California Institute of Technology, and through the Engineering and Physical Sciences Research Council (UK) grant number EP/H023348/1 for the University of Cambridge Centre for Doctoral Training, the Cambridge Centre for Analysis.  The authors benefited from fruitful discussion with Jean Dolbeault and Ivan Gentil about Proposition \ref{lem:poincare}. FH is grateful to Camille, Marine and Constance Bichet, and Joachim Schmitz-Justen and Rita Zimmermann for their hospitality during the SARS-CoV-2 outbreak that allowed to finish this project.


\nocite{*}
\bibliography{HOFFMANN-STABILITY-BACTERIA}
\bibliographystyle{acm}

\appendix 

\section{Reformulation of the improved Poincar\'e inequality}\label{sec:poincareproof}

\begin{lemma}
For $\lambda\ge \chi>0$, define $\Omega_{\lambda,\chi}:\R\times\R\to\R$ by 
\begin{equation}\label{def:OmegaApp}
\Omega_{\lambda,\chi}(x,y) =
\begin{cases}
 \left(M_\lambda(x) - M_\chi(x)\right)\left ( M_\lambda(y)  - M_\chi(y)\right ) + (1 - M_\chi(y))  M_\chi(x) & \quad \text{if $x\le y$}\,,\\
\left ( M_\lambda(x) - M_\chi(x)\right )\left ( M_\lambda(y)  - M_\chi(y)\right ) + (1 - M_\chi(x))  M_\chi(y) & \quad \text{if $x>y$}\,,
\end{cases}
\end{equation}
where $M_\lambda$ denotes the cumulative density function, $$M_\lambda(x) = \int_{z < x} \frac\lambda2 e^{-\lambda |z|} \,\rd z \,.$$
Then $\Omega_{\lambda,\chi}$ is non-negative and symmetric,
$$
\Omega_{\lambda,\chi}\ge 0\,,\qquad \Omega_{\lambda,\chi}(x,y)=\Omega_{\lambda,\chi}(y,x)\,,
$$
and we can rewrite the left-hand side of the Poincar\'e inequality \eqref{eq:poincare} as
\begin{equation*}
\frac12\int_{-\infty}^\infty |w(y) - \langle w \rangle_\lambda|^2 e^{- \chi |y|}\,\rd y 
 =   \iint  w'(x_1)   w'(x_2) \Omega_{\lambda,\chi}(x_1,x_2 )   \,\rd x_1 \,\rd x_2   \,.
\end{equation*}
\end{lemma}

\begin{proof}
We assume $\chi = 1$ without loss of generality and denote $\Omega_\lambda=\Omega_{\lambda,1}$. We can reformulate the left-hand-side of \eqref{eq:poincare} as follows:
\begin{align*}
 \int_{-\infty}^\infty |w(y) - \langle w \rangle_\lambda|^2 e^{-|y|}\,\rd y &= \int_{-\infty}^\infty \left ( \int_{-\infty}^\infty (w(y) - w(z)) \frac\lambda2 e^{-\lambda |z|} \,\rd z  \right )^2 e^{- |y|}\,\rd y 
\\
&= \int_{-\infty}^\infty \left ( \int_{-\infty}^\infty \left (\int_z^y   w'(x)\,\rd x\right ) \frac\lambda2 e^{-\lambda |z|} \,\rd z  \right )^2 e^{-  |y|}\,\rd y\\
& = \int_{-\infty}^\infty \left ( \int_{x <y}     w'(x) \left ( \int_{z < x} \frac\lambda2 e^{-\lambda |z|} \,\rd z \right ) \,\rd x \right. \\
& \qquad \qquad \left. - \int_{x>y}   w'(x) \left ( \int_{z > x}   \frac\lambda2 e^{-\lambda |z|} \,\rd z \right ) \,\rd x  \right )^2 e^{-  |y|}\,\rd y\,,
\end{align*} 
where we changed the order of integration to obtain the last line.
Denoting 
$$N_\lambda(x) := 1 - M_\lambda(x) = \int_{z > x} \frac\lambda2 e^{-\lambda |z|} \,\rd z$$
and expanding the last square, we find
\begin{align*}
\int_{-\infty}^\infty |w(y) - \langle w \rangle_\lambda|^2 e^{-  |y|}\,\rd y &=   \int_{-\infty}^\infty \left ( \int_{x <y}    w'(x)M_\lambda(x) \,\rd x  - \int_{x>y}   w'(x) N_\lambda(x) \,\rd x  \right )^2 e^{-  |y|}\,\rd y\\
& = \int_{-\infty}^\infty \left(  \iint  w'(x_1)    w'(x_2)  \tilde\Omega_\lambda(y,x_1,x_2)  \,\rd x_1 \,\rd x_2  \right) e^{-  |y|}\,\rd y\,,
\end{align*}
where\begin{multline*}
\tilde\Omega_\lambda(y,x_1,x_2) =  M_\lambda(x_1) M_\lambda(x_2)\ind_{x_1<y}\ind_{x_2<y} - M_\lambda(x_1) N_\lambda(x_2)\ind_{x_1<y}\ind_{x_2>y} \\
- N_\lambda(x_1) M_\lambda(x_2)\ind_{x_1>y}\ind_{x_2<y} + N_\lambda(x_1) N_\lambda(x_2)\ind_{x_1>y}\ind_{x_2>y}\, . 
\end{multline*}
Let us consider the case $x_1<x_2$ (the argument for $x_1>x_2$ is similar), and using the identity $N_\lambda = 1 - M_\lambda$, we can simplify the expression above:
\begin{align*}
\tilde\Omega_\lambda(y,x_1,x_2) &=  M_\lambda(x_1) M_\lambda(x_2)\ind_{x_2<y} - M_\lambda(x_1) (1 -M_\lambda(x_2))\ind_{x_1<y<x_2}\\
&\qquad+ (1 -  M_\lambda(x_1))(1-M_\lambda(x_2))\ind_{x_1>y}\, .
\end{align*}
By exchanging the order of integration, we find
\begin{align*}
\int_{-\infty}^\infty |w(y) - \langle w \rangle_\lambda|^2 e^{-  |y|}\,\rd y 
& =  2 \iint   w'(x_1)    w'(x_2) \Omega_\lambda(x_1,x_2 )   \,\rd x_1 \,\rd x_2   \,,
\end{align*}
where 
\begin{align*}
\Omega_\lambda(x_1,x_2) &= \frac12 \int \tilde\Omega_\lambda(y,x_1,x_2) e^{-  |y|}\,\rd y \\
& = 
 M_\lambda(x_1) M_\lambda(x_2) (1 - M_1(x_2)) - M_\lambda(x_1) (1 -M_\lambda(x_2)) (M_1(x_2) - M_1(x_1)) \\
 &\quad  + (1 -  M_\lambda(x_1))(1-M_\lambda(x_2))M_1(x_1)\\
 & = M_\lambda(x_1) M_\lambda(x_2)  - M_\lambda(x_1) M_1(x_2) + (1 - M_\lambda(x_2))  M_1(x_1) 
\end{align*}
if $x_1<x_2$, and $\Omega_\lambda(x_1,x_2):=\Omega_\lambda(x_2,x_1)$ if $x_1>x_2$. The symmetric expression \eqref{def:OmegaApp} for $\Omega_\lambda$ immediately follows.

Next, we show $\Omega_\lambda \geq 0$. It is enough to show non-negativity for $x<y$ by symmetry. By direct calculation,
$$
\frac{\rd}{\rd \lambda}M_\lambda(x)=\frac{x}{2}e^{-\lambda |x|}\,.
$$
Suppose on the one hand that $y>0$, then $\Omega_\lambda(x,y) \geq  M_\lambda(x) ( M_\lambda(y)  -  M_1(y)  ) \geq 0 $, because $\lambda \mapsto M_\lambda(y)$ is increasing for $y>0$. Suppose on the other hand that $y<0$, hence $x<0$ as well (recall $x<y$ by assumption).
Therefore, 
\begin{align*}
\Omega_\lambda(x,y) &\geq M_\lambda(x) M_\lambda(y)  - M_\lambda(x) M_1(y) + (M_1(y) - M_\lambda(y))  M_1(x) \\
&= \left(  M_\lambda(x) - M_1(x)\right)\left(M_\lambda(y) - M_1(y)\right)\,, 
\end{align*}
which is non-negative because $\lambda \mapsto M_\lambda(x)$ is decreasing for $x<0$. This concludes the proof.
\end{proof}

\section{$L^1$-stability in the case $\alpha=0$}\label{sec:alpha0}

In the case $\alpha=0$ dynamical arguments together with known stability results for shock waves \cite{Serre} provide global $L^1$ stability of solutions. We sketch the argument here for a more general signal response function $\phi:\R\to\R$ with the biologically reasonable assumptions that $\phi$ is odd, bounded, non-increasing and regular enough. 
For $u[\partial_xS]$ denoting the chemotactic flux, we consider the slightly more general bacterial chemotaxis model
\begin{subequations}
\label{quasimodel-app}
\begin{align}
& \partial_t \rho(t,x) =  \partial^2_{x} \rho(t,x) +  \, \partial_x \left(\rho(t,x)  u[\partial_x S]\right)\label{quasimodel1-app} \,,\\
& - \partial^2_{x} S(t,x)   = \rho(t,x)\,  .\label{quasimodel2-app}
\end{align}
\end{subequations} 
Given a symmetric velocity set $V\subset \R$, the macroscopic flux $u$ can be related to the microscopic behavior of individual cells via the signal response function, 
$$
u[\partial_x S](t,x)=\frac{1}{|V|} \int_{v\in V} v\phi(v\partial_x S(t,x))\rd v\,,
$$
an expression that was derived rigorously from a mesoscopic model in \cite{saragosti_mathematical_2010} by means of parabolic scaling techniques. 
Choosing the stiff response function $\phi(x)=-\sign(x)$ yields model \eqref{quasimodel} considered in this work (with $\chi:= \int_{v\in V} |v|\,\rd v/|V|$).

What allows us to handle the case $\alpha=0$ with an alternative method is a reformulation of system \eqref{quasimodel-app} as a viscous scalar conservation law (SCL). Substituting $\rho=- \partial^2_{x} S$ into \eqref{quasimodel1-app}, denoting
$$
z(t,x):= \partial_x S(t,x)\,,\qquad
f(r):=- \frac{1}{|V|}\int_{v\in V} \Phi(v r)\,\rd v
$$
with $\Phi$ the antiderivative of the signal response function $\phi$, and integrating \eqref{quasimodel1-app} once w.r.t. $x$, we obtain the viscous SCL
\begin{equation}\label{SCL}
\partial_t z + \partial_x f(z) = \partial_x^2 z\,.
\end{equation}
The  fundamental solution of \eqref{quasimodel2-app} is given by $K_0(x)=-|x|/2$, and since we fixed the cell mass at $\int \rho\,\rd x = 2/\chi$, 
we obtain the far-field conditions for any $t>0$,
\begin{equation}\label{BC}
z_\pm := \lim_{x\to\pm \infty} z(t,x)= \lim_{x\to\pm \infty} -\frac{1}{2}\int_\R \sign(x-y)\rho(t,y)\,\rd y = \mp \frac{1}{\chi}\,.
\end{equation}
Any stationary state $Z_\infty=Z_\infty(x)$ to \eqref{SCL}-\eqref{BC} satisfies (after another integration in $x$),
\begin{equation}\label{SCL-stat}
Z_\infty'=f(Z_\infty)-f(1/\chi)\,, \qquad \lim_{x\to\pm \infty} Z_\infty=\mp 1/\chi\,.
\end{equation}
where the integration constant is determined by the far-field condition and using the fact that $\Phi$ is even.
Whilst no explicit stationary state can be found, existence follows from an implicit argument.

\begin{proposition}[Existence]
Let $V\subset \R$ be the velocity set (which is a symmetric open interval), and let $\phi:\R\to\R$ be odd, non-increasing and in $C^{-1}(\R)$. Then system \eqref{quasimodel-app} admits at least one stationary state $(\rho_\infty, S_\infty)$.
\end{proposition}

\begin{proof}
Fix $x\in \R$. Applying the mean value theorem to the continuous function $g:\bar V\to \R$ defined as $g(v):=\Phi(vZ_\infty(x))-\Phi(v/\chi)$, there exists a velocity $v^*\in V$ such that \eqref{SCL-stat} can be reformulated as 
\begin{equation*}
Z_\infty'=F(Z_\infty):=\Phi(v^*/\chi)-\Phi(v^* Z_\infty)\,.
\end{equation*}
W.l.o.g. set $\Phi(0)=0$, and so $\Phi\le 0$ and $v^*\neq 0$.
$F$ is a convex function of $Z_\infty$ (since $\phi$ is assumed non-increasing) with the two zeros $z_+$ and $z_-$, and it follows that the above system has two equilibria, $z _+=-1/\chi$ being an attractor, and $z_-=1/\chi$ a repeller. For $-1/\chi < Z_\infty < 1/\chi$, the derivative $Z_\infty'$ is negative and so has the desired behavior ensuring that the far-field condition is satisfied. As the two stationary points lie on the same trajectory in phase space, a qualitative analysis of the phase portrait provides existence of a solution $Z_\infty$ satisfying \eqref{SCL-stat}.
\end{proof}

Using the reformulation as a viscous SCL above, we can show asymptotic stability with respect to the $L^1$-distance.

\begin{theorem}[$L^1$ stability]\label{thm:SCL}
Let $\phi\in C^1(\R)$.
Given a stationary state $Z_\infty\in L^1(\R)$ satisfying \eqref{SCL-stat}, consider an initial datum $z_0\in L^\infty(\R) \cap \left(L^1(\R)+Z_\infty\right)$ satisfying $z_+\le z_0(x) \le z_-$ for a.e. $x\in \R$. Then the solution $z(t,x)$ to \eqref{SCL}-\eqref{BC} satisfies
\begin{equation*}
\| z(t,\cdot)-Z_\infty(\cdot - h)\|_1 \to 0 \quad \text{ as } \quad t \to \infty
\end{equation*}
for the shift
$$
h:= \frac{\chi}{2}\int \left(z_0-Z_\infty\right)\,\rd x\,.
$$
\end{theorem}

This result yields more than just stability in the sense that a solution $z(t,x)$ initially close to the stationary state $Z_\infty$ remains close for all times. In fact, Theorem~\ref{thm:SCL} holds for any initial perturbation, providing global $L^1$-stability.

The proof of Theorem~\ref{SCL} makes use of the linear $\mathcal{C}^0$-semigroup $T_t:L^\infty(\R)\to L^\infty(\R)$ corresponding to the Cauchy problem for \eqref{SCL}, sending an initial datum $z_0\in L^\infty(\R)$ to its corresponding solution $z(t,x)$ with $z(0,x)=z_0(x)$. If $f$ in \eqref{SCL} is at least $C^2$, the semigroup $(T_t)_{t\ge 0}$ enjoys the four \emph{Co-Properties}: comparison ($a\le b$ a.e. implies $T_ta\le T_tb$ a.e), contraction ($\|T_ta-T_tb\|_1\le \|a-b\|_1$), conservation ($\int T_t a\,\rd x = \int a\,\rd x$) and constants (if $a$ is constant, then $T_ta=a$). These properties follow from the fact that the flux term $\partial_x f(z)$ in \eqref{SCL} may be interpreated as a lower order perturbation of the heat equation $\partial_t z = \partial_x^2 z$, see \cite{Serre}. The contraction property immediately yields the decay
$$
\frac{d}{dt} \|T_t z_0 - Z_\infty(\cdot -h)\|_1 \le 0
$$
and it remains to show that the limit is indeed zero as claimed.

\begin{proof}
The key approach is to first focus on the case where the initial datum lies between two shifts of the profile $Z_\infty$ as this case can be treated by means of dynamical systems theory using compactness of trajectories, $\omega$-limits and Lasalle's invariance principle (Theorem 3 in \cite[Chapter 7, Section 3.2]{Serre}, which relies on Lemma 4 in the same section). The general idea for this result is due to Osher and Ralston, see \cite{OsherRalston}. It then remains to show that the set of initial conditions considered in Theorem~\ref{thm:SCL} is included in the $L^1$-closure of the set of $L^\infty$ functions sandwiched between two arbitrary shifts of $Z_\infty$.  Indeed, for $z_0\in L^1 + Z_\infty$ such that $z_+\le z_0\le z_-$, the integral $\int_y^\infty |z_0(x)-Z_\infty(x)|\,\rd x$ vanishes as $y\to\infty$, and similarly for the distance towards $-\infty$. So for any $\eps>0$ there exist $s,t\in\R$ such that 
$$
\int_s^{+\infty}  |z_0(x)-Z_\infty(x)|\,\rd x <\frac{\eps}{3}\,,\qquad
\int_{-\infty}^t  |z_0(x)-Z_\infty(x)|\,\rd x <\frac{\eps}{3}\,.$$
For $\eta>0$ small enough such that $\eta<\eps/(3\chi |s-t|^{-1})$, we define
$\tilde z (x) $ equal to $(1-\chi \eta)z_0(x)$ on $(t,s)$ and equal to $Z_\infty$ elsewhere. Then
$$
\int  |\tilde z(x)-z_0|\,\rd x 
= \left( \int_s^{+\infty}+ \int_{-\infty}^t  \right)   |z_0(x)-Z_\infty(x)|\,\rd x + \chi \eta \int_t^s \,\rd x < \eps\,.
$$
And for big enough $k\in \R$ and small enough $j\in \R$, we have
$$
Z_\infty(x-j)\le \tilde z(x)\le Z_\infty(x-k)\qquad \text{ for a.e. } x\in\R\,.
$$
This concludes the proof. See also Corollary 1 in \cite[Chapter 7, Section 3.2]{Serre}.
\end{proof}

\end{document}